\documentclass[10pt]{article}

\usepackage{amssymb,amsmath,amsfonts,amsthm, mathrsfs}

\usepackage{todonotes}

\usepackage{nicefrac}

\usepackage{cancel}
\usepackage{bbold}

\usepackage[all]{xy}
\usepackage{graphicx}
\usepackage{epsfig}
\usepackage{makeidx}
\input{epsf.tex}
\input xy
\xyoption{all}
\usepackage{color}

\usepackage{tikz}
\usepackage{verbatim}

\usepackage{pdfpages}

\definecolor{darkgreen}{HTML}{006622}

\usepackage[margin = 1em]{subfig}
\usepackage{geometry}
\usepackage{pgfplots}
\pgfplotsset{compat = 1.11}
\usetikzlibrary{calc,decorations.pathmorphing,decorations.pathreplacing,patterns}
\usepgfplotslibrary{colormaps} % used for the grayscale contour pictures
\usetikzlibrary{arrows.meta}
\tikzset{>={Latex[width=1.7mm,length=2.2mm]}}
\usetikzlibrary{decorations.text}

\tikzset{zigzag/.style={decorate, decoration=zigzag}}

\makeatletter
\def\@hex@@Hex#1%
 {\if a#1A\else \if b#1B\else \if c#1C\else \if d#1D\else
  \if e#1E\else \if f#1F\else #1\fi\fi\fi\fi\fi\fi \@hex@Hex}
\makeatother

\def\E{{\mathbb E}}
\def\P{{\mathbb P}}
\def\I{{\mathbb{1}}}
\def\phim{{\varphi(x,t)}}
\def\phimx{{\varphi'_x(x,t)}}
\def\phimt{{\varphi'_t(x,t)}}

\newtheorem{thm} {\noindent \bf{Theorem}}[section]

\newtheorem{remark}{\noindent \bf{Remark}}[section]

\newtheorem{cor}{\noindent \bf{Corollary}}[section]

\newcommand{\eqn}[1]{\begin{equation}#1\end{equation}}
\newcommand{\eqan}[1]{\begin{align}#1\end{align}}
\newcommand{\asymptotics}[1]{}
\newcommand{\nn}{\nonumber}

\begin{document}

\title{Unified approach for solving
exit problems for
%multiplicative growth-collapse
additive-increase and
multiplicative-decrease processes}
\author{Remco van der Hofstad\footnotemark[4]\and
        Stella Kapodistria\footnotemark[1]
        \and Zbigniew Palmowski\footnotemark[2]
        \and Seva Shneer\footnotemark[3]
        }

\date{\today}

\maketitle

\footnotetext[1]{Department of Mathematics and Computer Science, Eindhoven University of Technology,
P.O.\ Box 513, 5600 MB  Eindhoven, The Netherlands,
E-mail:   \texttt{s.kapodistria@tue.nl}}
\footnotetext[3]{Heriot-Watt University, Edinburgh, UK
E-mail:   \texttt{v.shneer@hw.ac.uk}}
\footnotetext[2]{Faculty of Pure and Applied Mathematics, Wroclaw University of Science and Technology, Wyb. Wyspia\'nskiego 27, 50-370 Wroclaw, Poland,
E-mail:   \texttt{zbigniew.palmowski@gmail.com}}
\footnotetext[4]{Department of Mathematics and Computer Science, Eindhoven University of Technology,
P.O.\ Box 513, 5600 MB  Eindhoven, The Netherlands,
E-mail:   \texttt{r.w.v.d.hofstad@tue.nl}}

\begin{abstract}
We analyse an additive-increase and multiplicative-decrease (aka growth-collapse) process that grows linearly in time and that  experiences downward jumps at Poisson epochs that are (deterministically) proportional to its present position.
This process is used for example in modelling of Transmission Control Protocol (TCP) and can be viewed as a particular example of the so-called shot noise model, a basic tool in modeling earthquakes, avalanches and neuron firings.

For this process, and also for its reflected versions, we consider one- and two-sided exit problems that concern the identification of the laws of exit times from fixed intervals and half-lines. All proofs are based on a unified first-step analysis approach at the first jump epoch, which allows us to give explicit, yet involved, formulas for their Laplace transforms.

All the eight Laplace transforms can be described in terms of
%four
two so-called {\em scale functions}
%$Z_{\uparrow},  Z_{\downarrow}, L_{\uparrow}$ and $L_{\downarrow}$, which can be reduced to
$Z_{\uparrow}$ and $L_{\uparrow}$.
%due to relations between them.
Here $Z_{\uparrow}$ is described in terms of multiple explicit sums, and $L_{\uparrow}$ in terms of an explicit recursion formula. All other Laplace transforms can be obtained from $Z_{\uparrow}$ and $L_{\uparrow}$ by taking limits, derivatives, integrals and combinations of these. %Thus, as for L\'evy processes, we only require {\em two} scale functions.

%Results are relevant for

\asymptotics{We also give the Cram\'er-type asymptotics for the one-sided downward probability when the initial position tends to infinity.}
\medskip

\end{abstract}

\section{Introduction}
\label{sec-intro}
We analyse an additive-increase and multiplicative-decrease (aka growth-collapse or stress-release) process $X\equiv\left ( X_t\right)_{t\geq 0}$ that grows linearly with slope $\beta$ and experiences downward jumps at Poisson epochs, say $(T_i)_{i\in\mathbb{N}}$ with fixed intensity $\lambda$. The collapses are modeled by multiplying the present process position by a fixed proportion $p\in (0,1)$, that is $-\Delta X_{T_i}=(1-p)X_{T_i-}$ for $\Delta X_t=X_t-X_{t-}$. We assume that the process starts on the positive half-line $X_0=x>0$. An illustration of a path of the process $X_t$ is depicted in Figure \ref{fig: Fluid_model}. For more information on this class of processes, the interested reader is referred to \cite{LLO}. Note that, without loss of generality, we can assume that $\beta=1$. Results for general $\beta$ may be obtained using a simple time rescaling.
%\todo[inline, color=green]{@All: The notation has been adapted $Q\to p$, $q\to w$, and all stochastic processes are now $(X_t)_{t\geq 0}$. Please %check if the new notation is everywhere consistent.}

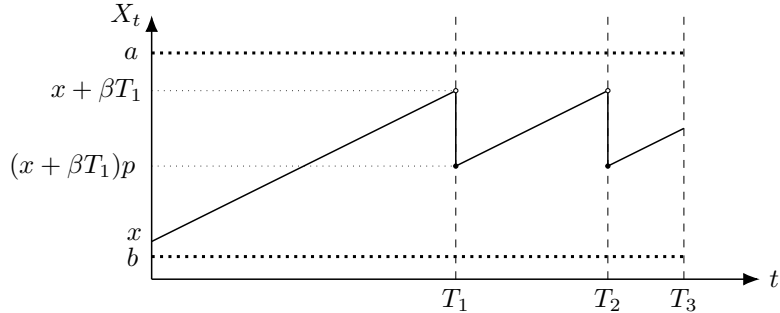
\begin{figure}[h!]
\begin{center}
\begin{tikzpicture}
\draw[line width=0.2mm,black,->] (0,-0.5) --(8,-0.5) node[right,align=left]     {$t$};
\draw[line width=0.2mm,black,->] (0,-0.5) --(0,3) node[left,align=left]     {$X_t$};
\draw[line width=0.2mm,black]  (0,0) -- (4,2)  -- (4,1) -- (6,2) -- (6,1) -- (7,1.5); %slope=0.5
\draw[line width=0.05mm,black,dashed]  (4,-0.5) -- (4,3); \node [below] at (4,-0.5) {\textcolor{black}{$T_1$}};
\draw[line width=0.05mm,black,dashed]  (6,-0.5) -- (6,3); \node [below] at (6,-0.5) {\textcolor{black}{$T_2$}};
\draw[line width=0.05mm,black,dashed]  (7,-0.5) -- (7,3); \node [below] at (7,-0.5) {\textcolor{black}{$T_3$}};
\draw[line width=0.4mm,black,dotted]  (0,-0.2) -- (7,-0.2); \node [left] at (-0.05,-0.2) {\textcolor{black}{$b$}};
\draw[line width=0.4mm,black,dotted]  (0,2.5) -- (7,2.5); \node [left] at (-0.05,2.5) {\textcolor{black}{$a$}};
\node [right] at (-0.45,0.1)   {\textcolor{black}{$x$}};
\draw[line width=0.1mm,black,dotted]  (0,2) -- (4,2); \node [right] at (-1.45,2)   {\textcolor{black}{$x+\beta T_1$}};
\draw [fill=white, draw=black] (4,2) circle [radius=0.8pt];
\draw [fill=black, draw=black] (4,1) circle [radius=0.8pt];
\draw[line width=0.1mm,black,dotted]  (0,1) -- (4,1); \node [right] at (-2,1)   {\textcolor{black}{$(x+\beta T_1)p$}};
\draw [fill=white, draw=black] (6,2) circle [radius=0.8pt];
\draw [fill=black, draw=black] (6,1) circle [radius=0.8pt];
%\node [right] at (1.6,0.7)   {\textcolor{black}{(slope $\beta$)}};
\end{tikzpicture}
\end{center}\caption{Additive-increase
and multiplicative-decrease process path\label{fig: Fluid_model}}
\end{figure}

The additive-increase and the multiplicative-decrease process has various applications. For example, it appears as the fluid limit scaling for some queueing models (with binomial catastrophe rates) used in modeling population growth subject to mild catastrophes, see, e.g.,  \cite{Adan-Economou-Kapodistria-2009, Artalejo} and the references therein.  Such processes can be viewed as a particular example of the so-called shot noise model, which is used in models of earthquakes, avalanches or neuron firings. Moreover, this process is also used in the AIMD algorithm to model the Transmission Control Protocol (with $p=\tfrac{1}{2}$), the dominant protocol for data transfer over the internet (see \cite{15, Bert}).

\paragraph{Main contribution of the paper.}
The main objective of this work is to identify the laws of the exit (aka first passage) times
	\begin{equation}
	\label{exittimes_1}
	\tau_{\uparrow a}(x) = \inf\{t\geq 0\colon X_t> a \mid X_0=x\},
	\end{equation}
for $x\in[0, a),%0\le x < a
$ and
	\begin{equation}
	\label{exittimes_2}
	\tau_{\downarrow b}(x)=\inf\{t \geq 0\colon X_t\leq  b\mid X_0=x\}
	\end{equation}
for $x\in [b,+\infty)$, % \ge b > 0$
and to present a unifying framework for their derivation.
%Apart of solving exit problems we also identify respective resolvents of the process $X$
%killed on exiting from interval or half-lines. We identify the undershoot law as well.
%We start our analysis from analyzing one-sided downward exit problem by
%construction the recursive equation for a bivariate Laplace transform (with respect of time and
%initial position) of the tail of distribution of $\tau_a^-$.
%This allows us to identify the law of $\tau_a^-$.
%Later, in this text, we rely mainly on the Strong Markov property to solve
%other exit problems.

The master plan, in a fluctuation theory of Markov processes with jumps in one direction, is to produce a great variety of exit identities
%or first passage problems
in terms of a few key functions only (the so-called {\em scale functions}, a term originating from diffusion theory; see e.g.\ \cite{Feller1, Feller2, Feller3, ItoMcKean}). These crucial functions appear %amongst others
in the Laplace transform of the exit times in \eqref{exittimes_1} and \eqref{exittimes_2}. Hence the first main step is to identify these `alphabet functions'. This is the main result in this paper, where we identify these scale functions. It is commonly believed that one needs at most three scale functions (or three letters) only.

In the case of L\'evy processes there are only {\em two} scale functions %, $W$ and $Z$,
which are related with the various ways in which the process can exit the interval: only in a continuous way via the upper end of this interval and possibly by a jump via the lower end of this interval; see \cite{Avrametalrisks} for further discussion. With these scale or key functions at hand,  we are usually able to solve most other identities for the original process or for related transformed ones. Such transformations usually are obtained by a {\em reflection} (at the running infimum or supremum of the process), a refraction (only fixed portion of the process is reflected), a twisting of measure, or an additional randomization (via Poisson observation, subordination etc). Finally, this `scale-functions paradigm' is used in many applications appearing in queueing theory, actuarial science, optimization, mathematical finance or control theory, again exemplifying their importance; see Kyprianou \cite{Andreas} for details.

The above complete plan has been executed only for spectrally negative L\'evy processes - an overview of this theory is given in the book of Kyprianou \cite{Andreas}; see also \cite{AKP, Bertoin1, Bertoin2, Bingham, Doney, Emery, KP, Rogers, Zolotarev}. Most proofs in this theory only rely on two key properties: the Markov property and the skip-free property. Hence, there is hope that part of this master plan can be realized for other processes with downward jumps only as well. In this paper, we identify `alphabet' functions for an additive-increase and multiplicative-decrease process $X$ introduced formally above.

We show that for the additive-increase and multiplicative-decrease process, as well as for related transformed ones, we only need {\em two} scale functions to which all other Laplace transforms can be related, see Remarks \ref{rem:two_sided_up} and \ref{rem:two_sided_down} below.

There are already some results of this kind for other Markov processes; see  \cite{BG} for an overview. In risk and queueing theories some of the L\'evy-type results have been already generalized to the compound renewal processes; see \cite{AsmussenAlbrecher}. Other deep results have been obtained for diffusion processes (see \cite{diff1, diff2, diff3, 22, Lehoczky, diff4}). Similar results have been derived in the context of generalized spectrally negative Ornstein-Uhlenbeck processes (see \cite{Hadjiev, JJ, Novikov, Patie08}). Later, spectrally positive self-similar Markov processes were analyzed as well; see \cite{Liandme}. Other types of processes where scale functions have been successfully identified concern: L\'evy-driven Langevin-type processes (see \cite{czarnaetal}); affine processes (see \cite{AvramUsabel}); Markov addititive processes (see \cite{Jevme, KP2}); Segerdahl-Tichy processes (see \cite{AvramJose, ewame, PaulsenandGjessing, seger, tichy}).

The additive-increase and multiplicative-decrease processes under consideration have one substantial difference from the above mentioned L\'evy processes, namely the jump size depends on the position of the process $X$ prior to this jump. This produces substantial difficulties in solving the exit problems and handling them requires a new approach. In particular, the principal tools in L\'evy-type fluctuation theory, Wiener--Hopf factorization and It\^{o}'s excursion theory, are not available in our case. Instead, we rely on a {\em first-step analysis} that allows us to identify the two scale functions.
% for additive-increase and the multiplicative-decrease processes, which
This is a novel approach for such problems, as we explain in more detail in the sequel.

\paragraph{First-step analysis as a main method.}
The proposed %To overcome these difficulties we propose a
unifying approach for the computation of the exit identities relies on a first-step analysis based on finding the position of the considered processes right after the first jump epoch. This approach produces a recursive equation, which we subsequently solve. Instead of this approach, one could also implement the differential equation method, as often used for diffusion and L\'evy processes, and this would yield the same recursive equation as the first-step analysis proposed in this paper.

We believe that this approach could be used for other Markov processes having the skip-free property as well. Although this method has long been known in the literature (see e.g.\ \cite[Chap. XII]{AsmussenAlbrecher}, \cite[Part 2]{Cohen}),
we believe that it is the first time that this whole analysis (i.e., the identification of all scale functions) is done using only this method. To some extend we thus use only the Markov property and the structure of trajectories, whilst the classical tools of Wiener-Hopf factorization or martingale theory are not used.

We manage to solve the exit problems for reflected processes (where reflection occurs at the lower and upper fixed levels, as well as at the running infimum and maximum).
For such processes, our approach is more standard and it is based on using the method by Lehoczky \cite{Lehoczky}, and on the construction of a Kennedy martingale adapted to our model, followed by an  application of the optimal stopping theorem.
\medskip

%\begin{figure}[h!]
%\begin{center}
%\[
%\def\labelstyle{\scriptstyle}
% \xymatrix @C+3pc @R+2pc  {
%0\ar[r]^{\beta }
%& 1 \ar@/^1pc/[l]|{\binom{1}{1}(1-p)\lambda }\ar[r]^{\beta}
%& 2\ar@/^1pc/[l]|{\binom{2}{1}(1-p)p\lambda}
%	\ar@/^2.5pc/[ll]|{\binom{2}{2}(1-p)^{2}\lambda }\ar[r]^{\beta}
%& 3\ar@/^1pc/[l]|{\binom{3}{1}(1-p)p^{2}\lambda}
%	\ar@/^2.5pc/[ll]|(.5){\binom{3}{2}(1-p)^{2}p\lambda}
%	\ar@/^4pc/[lll]|{\binom{3}{3}(1-p)^{3}\lambda }\ar[r]^{1 }&\cdots&}
%\]
%\end{center}\caption{State-transition diagram\label{figure1}}
%\end{figure}

\asymptotics{\paragraph{Asymptotic results.}
Starting from the seminal paper of Cram\'er the asymptotics of the probability
that the first downward pasage time \eqref{exittimes_2} is finite for large initial
values of $X_0=x$ has been analyzed in many scientific works.
Based on analysis of our scale functions using Karamata theorem
we managed to derive the asymptotics of ...}

%\textcolor{orange}{After logarithmic space transformation, the new process  $\log X_t$  is a so-called piece-wise logarithmic Markov processes %moving along the logarithmic curve $\log(x+t)$ between additive negative jumps of fixed size $-\log p$. Hence, exit problems considered in this %paper for the original process $X_t$ correspond to exit problems for the process $\log X_t$ with transformed barriers $\log b$ and $\log a$.}

\paragraph{Organisation of the paper.}
The remainder of this paper is structured as follows.
In Sections \ref{sec:main1} and \ref{sec:main2} we present our main results with their proofs. %, and in Section \ref{sec-as} we apply our results to identify the asymptotical properties of the scale functions.
In Section \ref{sec-disc}, we close with a discussion of our results, alternative approaches, as well as possible future directions.

\section{Exit identities}
\label{sec:main1}
Given the initial position of the stochastic process, say $X_0=x$, the exit problems are solved by characterizing the Laplace-Stieltjes transforms of $\tau_{\uparrow a}(x)$ in \eqref{exittimes_1}, $\tau_{\downarrow b}(x)$ in \eqref{exittimes_2}, and
	\eqn{
	\label{tau-ab-def}
	\tau_{a,b}(x)=\min\{\tau_{\uparrow a}(x),\, \tau_{\downarrow b}(x)\},
	}
with $x\in (b,a)$. %b<x<a$.
Note that the stochastic process $X$ will hit the threshold $a$ when crossing upwards as it can only move continuously upwards (creeping). On the other hand, it jumps over the threshold $b$ when crossing $b$ from above. As such (creeping versus jumping), the derivation of the two exit times encompasses different properties requiring typically very different techniques. Here, we propose a unifying framework for exit times,  both for one- and for two-sided exit problems. In what follows, suppose that the law $\P_{x}$ corresponds to the conditional version of $\P$ given that $X_0=x$. Analogously, $\E_x$ denotes the expectation with respect to $\P_x$. Let $\mathcal{F}_t$ be a right-continuous natural filtration of $X$ satisfying the usual conditions.

We now state our main results, given in Theorems \ref{mainthm1} and \ref{mainthm2}, that study upward and downward crossing problems, and Theorems \ref{mainthm3} and \ref{mainthm4} that study two-sided exit problems, below. We start by discussing upward one-sided exit problems:

\begin{thm}[Upward one-sided exit problem]\label{mainthm1}
Given $x\in (0,a)$, the Laplace-Stieltjes transform of the upward exit time $\tau_{\uparrow a}(x)$ of the  additive-increase
and multiplicative-decrease process $\left ( X_t\right)_{t\geq 0}$ is given by
	\begin{equation}
	\label{upwardone1}
	Z_\uparrow (w;  x,a):=\E_x [\mathrm{e}^{-w \tau_{\uparrow a}(x)}]=\frac{Z_\uparrow (w; 0,a)}{Z_\uparrow (w; 0,x)}, \qquad\mathrm{Re}[w]>0,
	\end{equation}
with
	\begin{equation}
	\label{H}
	Z_\uparrow (w; 0,x)=\frac{1}{
	1+w  \sum_{n=1}^\infty\frac{x^n}{n!}
	\prod_{i=1}^{n-1}(w+\lambda - \lambda p^i)}.
	\end{equation}
\end{thm}

Note that this result matches the result reported in \cite[Section 4.2]{LL}, in which the authors instead use a martingale approach for the derivation of the recursive equation satisfied by the Laplace-Stieltjes transform of the exit time.

\begin{proof}
Note that
	$$
	\tau_{\uparrow a}(0) = \tau_{\uparrow x}(0) + \tau_{\uparrow a}(x)
	$$
for all $x\in(0,a)$, where $\tau_{\uparrow x}(0)$ and $\tau_{\uparrow a}(x)$ are independent. This is due to the strong Markov property, combined with the fact that $X_{\tau_{\uparrow a}(x)}=a$ since $X$ does not have upwards jumps (only creeps upwards). Hence,
	\begin{equation}
	\label{eq:thm1_aux1}
	Z_\uparrow (w; 0,a) = Z_\uparrow (w; 0,x) Z_\uparrow (w; x,a)
	\end{equation}
for all $q \in\mathbb{C}$ with $\mathrm{Re}[w]>0$, and for all $x\in(0,a)$. This implies that in order to prove Theorem \ref{mainthm1} it suffices to prove \eqref{H}.

For this purpose, we write
	\eqn{
	\tau_{\uparrow a}(0) \stackrel{d}{=} a \I_{\{T > a\}} + \left(T + \tau_{\uparrow a}(Tp) \right)\I_{\{T \le a\}},
	}
with  $T=\inf\{t\colon X_t<X_{t-}\}$ denoting the time of the first downward jump, which is exponentially distributed with intensity $\lambda$ and $\I_{\{\cdot\}}$ denoting the indicator function taking value one if the event inside the brackets is satisfied and zero otherwise. The above result readily implies that
	\begin{align*}
	Z_\uparrow (w; 0,a) & :=\E_0 [\mathrm{e}^{-w \tau_{\uparrow a}(0)}]= \mathrm{e}^{-(\lambda+w )a} + \int_{0}^{a} \lambda \mathrm{e}^{-\lambda t} \mathrm{e}^{-w t} \E_{tp} [\mathrm{e}^{-w\tau_{\uparrow a}(tp)}]\mathrm{d}t\\
	& = \mathrm{e}^{-(\lambda+w )a} + \lambda \int_{0}^{a} \mathrm{e}^{-(\lambda+w ) t} Z_\uparrow (w; tp,a) \mathrm{d}t.
	\end{align*}
In light of \eqref{eq:thm1_aux1}, this yields
	\begin{align*}
	Z_\uparrow (w; 0,a) & = \mathrm{e}^{-(\lambda+w )a} + \lambda Z_\uparrow (w; 0,a) \int_{0}^{a} \mathrm{e}^{-(\lambda+w ) t} \frac{1}{Z_\uparrow (w; 0,tp)} \mathrm{d}t.
	\end{align*}
The above may be rewritten as
	\begin{equation}
	\label{eq:thm1_aux3}
	1 = \frac{\mathrm{e}^{-(\lambda+w )a}}{Z_\uparrow (w; 0,a)} + \lambda \int_{0}^{a} \mathrm{e}^{-(\lambda+w ) t} \frac{1}{Z_\uparrow (w; 0,tp)} \mathrm{d}t.
	\end{equation}
Let $\tilde{Z}_\uparrow(w, s) = \int_0^\infty \mathrm{e}^{-sa} \frac{1}{Z_\uparrow (w; 0,a)} \mathrm{d}a$ (note that this is well defined for $s\in\mathbb{C}$ with $\mathrm{Re}[s]>\lambda+\mathrm{Re}[w]$). Multiplying both sides of \eqref{eq:thm1_aux3} by $\mathrm{e}^{-sa}$ and integrating over $a$ yields, after straightforward manipulations,
	\begin{align*}
	%\frac{1}{s} & = \tilde{Z}(w, \lambda+w+s) + \lambda \int_{a=0}^\infty \mathrm{e}^{-sa} \int_{z=0}^a \mathrm{e}^{-(\lambda+w ) z} \frac{1}{Z_{zp}(w, 0)} dz da \\
	%& = \tilde{Z}(w, \lambda+w+s) + \lambda \int_{z=0}^\infty \mathrm{e}^{-(\lambda+w ) z} \frac{1}{Z_{zp}(w, 0)} \int_{a=z}^\infty \mathrm{e}^{-sa} da dz\\
	%& = \tilde{Z}(w, \lambda+w+s) + \frac{\lambda}{s} \int_{z=0}^\infty \mathrm{e}^{-(\lambda+w+s) z} \frac{1}{Z_{zp}(w, 0)} dz\\
	%& = \tilde{Z}(w, \lambda+w+s) +\frac{\lambda}{sp} \int_{y=0}^\infty \mathrm{e}^{-\frac{\lambda+w+s}{p} y} \frac{1}{Z_{y}(w, 0)} dz\\
	%& = \tilde{Z}(w, \lambda+w+s) +\frac{\lambda}{sp}\tilde{Z}\left(w, \frac{\lambda+w+s}{p}\right).
	s\tilde{Z}_\uparrow(w, \lambda+w+s)& =- \frac{\lambda}{p}\tilde{Z}_\uparrow\left(w, \frac{\lambda+w+s}{p}\right)+1.
	\end{align*}
Setting $z = \lambda + w + s$  leads to the recursive relation
	$$
	(\lambda+w -z) \tilde{Z}_\uparrow(w, z)=  \frac{\lambda}{p} \tilde{Z}_\uparrow(w, zp^{-1})-1.
	$$
One can easily check that
	\begin{align*}
	\tilde{Z}_\uparrow(w, z)
	& =\sum_{n=0}^\infty \frac{1}{z^{n+1}} \prod_{i=0}^{n-1} (\lambda+w -\lambda p^i)\\
	&= \frac{1}{z} + w \sum_{n=1}^\infty \frac{1}{z^{n+1}} \prod_{i=1}^{n-1} (\lambda+w -\lambda p^i)
	\end{align*}
satisfies the above recursive relation. Alternatively, one can look for a solution in the form of $\tilde{Z}_\uparrow(w, z) = \sum_{n=-\infty}^\infty c_n(w) z^{-n}$
with unknown $c_n(w)$. The above relation readily implies that $c_n(w)=0$ for all $n\le 0$, $c_1(w)=1$ and, for $n \ge 2$,
	$$
	c_n(w) = w \prod_{i=1}^{n-1} (\lambda+w -\lambda p^i).
	$$
Equation \eqref{H} then follows immediately.
\end{proof}
\medskip

We continue by studying downward one-sided exit problems:

\begin{thm}[Downward one-sided exit problem]
\label{mainthm2}
Given $x\in (b,\infty)$, the Laplace-Stieltjes transform
	\begin{equation}\label{Zdownarrowdef1}
	Z_\downarrow(w;x,b):= \E_x[\mathrm{e}^{-w\tau_{\downarrow b}(x)}], \qquad\mathrm{Re}[w]>0,
	\end{equation}
of the downward crossing time $\tau_{\downarrow b}(x)$ of the additive-increase and multiplicative-decrease process $\left ( X_t\right)_{t\geq 0}$ equals
%, for $z\in\mathbb{C}$ with $\mathrm{Re}[z]>\lambda+\mathrm{Re}[w]$ and for $q\in\mathbb{C}$ with $\mathrm{Re}[w]>0$,
	\begin{align}
	Z_\downarrow(w;x,b)&=\frac{w}{w+\lambda}
	\sum_{k=1}^\infty
	\left(\frac{\lambda}{w+\lambda}\right)^k
	\mathbb{1}_{\{b<x\leq bp^{-k}\}}\nonumber\\
	&\quad+
	\frac{w}{\lambda(w+\lambda)}
	\sum_{k=0}^\infty
	\left(\frac{\lambda}{w+\lambda}\right)^{k+1}
	 \sum_{i=0}^{k}
	 \frac{
	 1-(1+\tilde{C}(w;b))p^{i-k}
	 }{ \prod\limits_{j=0,j\neq i}^{k}(1-p^{i-j})}
	\mathbb{1}_{\{bp^{-k}> x\}}\mathrm{e}^{(w+\lambda)p^i x},
	\label{zdownarrowformula}
	\end{align}
with
\begin{align}
\tilde{C}(w;b)&=
%	\frac{
%	\sum\limits_{l=1}^\infty
%	\mathrm{e}^{-b (w+\lambda) p^{-l}}
%	\frac{\lambda^{l}}{ (w+\lambda)^{l}\prod\limits_{i=1}^{l-1}(1- p^{-i})}
%	}
%	{\sum\limits_{l=0}^\infty
%	\mathrm{e}^{-b(w+\lambda)p^{-l} }\frac{\lambda^{l}}{p^{l} (w+\lambda)^{l}\prod\limits_{i=1}^{l}(1-p^{-i})}
%	}\\
%	&=\frac{
%	\sum\limits_{l=0}^\infty
%	\mathrm{e}^{-b (w+\lambda) p^{-l}}
%	\frac{\lambda^{l}(1-p^{-l})}{ (w+\lambda)^{l}\prod\limits_{i=1}^{l}(1- p^{-i})}
%	}
%	{\sum\limits_{l=0}^\infty
%	\mathrm{e}^{-b(w+\lambda)p^{-l} }\frac{\lambda^{l}p^{-l}}{(w+\lambda)^{l}\prod\limits_{i=1}^{l}(1-p^{-i})}
%	}\\
%	&=
	\frac{
	\sum\limits_{l=0}^\infty
	\mathrm{e}^{-b (w+\lambda) p^{-l}}
	\frac{\lambda^{l}}{ (w+\lambda)^{l}\prod\limits_{i=1}^{l}(1- p^{-i})}
	}
	{\sum\limits_{l=0}^\infty
	\mathrm{e}^{-b(w+\lambda)p^{-l} }\frac{\lambda^{l}}{p^{l}(w+\lambda)^{l}\prod\limits_{i=1}^{l}(1-p^{-i})}
	}-1.\label{Eq:ConstantCLSTDown}
\end{align}	
\end{thm}

\begin{proof}
In order to compute the first downward crossing time, we employ a first-step analysis approach yielding
	\begin{align} \label{eq:tau_down}
	\tau_{\downarrow b}(x)\stackrel{d}{=}
	T\mathbb{1}_{\{(x  + T)p\leq b\}}+
	\left(T+\tau_{\downarrow b}((x  + T)p)
	\right)
	\mathbb{1}_{\{(x  + T)p> b\}},
	\end{align}
with  $T$ denoting the time of the first downward jump, which is exponentially distributed with intensity $\lambda$. Let
\[\tilde{Z}_\downarrow(w,z;b):=\int_{b}^\infty \mathrm{e}^{-z x }\mathbb{E}_x [\mathrm{e}^{-w\tau_{\downarrow b}(x)}]\mathrm{d}x,\quad \text{with }\mathrm{Re}[z]>\lambda+\mathrm{Re}[w],\]
then, the above result, after cumbersome but straightforward computations, implies that
	\begin{align}
	%See page 9 of fluid-limit_v2 in the folder Notes_Stella
	\tilde{Z}_\downarrow (w, z; b)
	%	&=\frac{\lambda}{w+\lambda }\left(\frac{  \mathrm{e}^{-\frac{b  (w+\lambda )}{p}} \left(\mathrm{e}^{b  (w+\lambda-z)}-\mathrm{e}^{\frac{b  (w+\lambda-z) )}{p}}\right)}{w+\lambda-z }
	%+\frac{\mathrm{e}^{-b z }-\mathrm{e}^{-\frac{b z}{p}}}{z}\right)\\
	%	&\quad -\frac{\lambda \mathrm{e}^{b(w+\lambda- z ) }}{p (w+\lambda-  z )}Z_\downarrow (w, (w+\lambda)/p; b)
	%	+\frac{\lambda}{p (w+\lambda-  z )}Z_\downarrow (w, z/p; b)\\
	&=\frac{\lambda}{w+\lambda }\left(\frac{  \mathrm{e}^{b  (w+\lambda)(1-p^{-1})-bz}-\mathrm{e}^{-bz p^{-1}}}{w+\lambda-z }+\frac{\mathrm{e}^{-b z }-\mathrm{e}^{-b z p^{-1}}}{z}\right)\nonumber\\
	&\quad -\frac{\lambda \mathrm{e}^{b(w+\lambda- z ) }}{p (w+\lambda-  z )}\tilde{Z}_\downarrow (w, (w+\lambda)p^{-1}; b)
	+\frac{\lambda}{p (w+\lambda-  z )}\tilde{Z}_\downarrow (w, z p^{-1}; b).
	\label{Eq:recursion_LST_DC}
\end{align}

All in all, the above can be written as	
	\begin{align}\label{recursion}
	\tilde{Z}_\downarrow (w, z; b) &=C(w, z)+D(w, z)\tilde{Z}_\downarrow (w, z p^{-1}; b)
	\end{align}
yielding upon iterating
	\begin{align*}
	\tilde{Z}_\downarrow (w, z; b) &=\sum_{k=0}^\infty C(w, z p^{-k})\prod_{i=0}^{k-1}D(w, z p^{-i})+\lim_{k\to\infty}\tilde{Z}_\downarrow (w, z p^{-k};b)\prod_{i=0}^{k-1}D(w, z p^{-i}).
	\end{align*}
Note that $\lim_{k\to\infty}\tilde{Z}_\downarrow (w, z p^{-k};b)=0$
and that
$\lim_{k\to\infty}\prod_{i=0}^{k-1}D(w, z p^{-i})=0$
 for all $z\in\mathbb{C}$ with $\mathrm{Re}[z]>\lambda+\mathrm{Re}[w]$. 	All in all,
	\begin{align}
	\tilde{Z}_\downarrow (w, z; b) &=\frac{\lambda}{w+\lambda}
	\sum_{k=0}^\infty
	\left(\frac{
	\mathrm{e}^{b  (w+\lambda)(1-p^{-1}) -  bz p^{-k}}-\mathrm{e}^{- b z p^{-k-1}}}{w+\lambda -  z p^{-k}}
	+\frac{\mathrm{e}^{-b z p^{-k}}-\mathrm{e}^{- b  z p^{-k-1}}}{z p^{-k}}\right)
	\prod_{i=0}^{k-1}\frac{\lambda}{p (w+\lambda -  z p^{-i})}\nonumber\\
	&\qquad -\tilde{Z}_\downarrow (w, (w+\lambda ) p^{-1};b)\sum_{k=0}^\infty  \mathrm{e}^{b(w+\lambda - z p^{-k}) }
	\prod_{i=0}^{k}\frac{\lambda}{p (w+\lambda -  z p^{-i})}.
	\label{Eq:Predownone}
	\end{align}
%\begin{align}
%\tilde{Z}_\downarrow (w, z; b)&=\frac{\lambda}{w+\lambda}
%\sum_{k=0}^\infty
%\left(\frac{
%\mathrm{e}^{b  (w+\lambda)(1-p^{-1}) -  bz p^{-k}}-\mathrm{e}^{- b z p^{-k-1}}
%}{w+\lambda -  z p^{-k}}+\frac{\mathrm{e}^{-b z p^{-k}}-\mathrm{e}^{- b  z p^{-k-1}}}{z p^{-k}}\right)
%	\prod_{i=0}^{k-1}\frac{\lambda}{p (w+\lambda -  z p^{-i})}\nonumber\\
%	&\qquad -\tilde{Z}_\downarrow (w, (w+\lambda ) p^{-1};b)\sum_{k=0}^\infty  \mathrm{e}^{b(w+\lambda - z p^{-k}) }
%	\prod_{i=0}^{k}\frac{\lambda}{p (w+\lambda -  z p^{-i})}.\label{Eq:LS_Sol}
%\end{align}
In order to compute $\tilde{Z}_\downarrow (w, (w+\lambda ) p^{-1}; b)$, we first multiplying \eqref{Eq:Predownone} with $w+\lambda - z $. After simplifying the resulting expressions, we set $z =w+\lambda$ rendering the LHS of \eqref{Eq:Predownone} zero. This yields after some straightforward algebraic manipulations
\begin{align}
	\tilde{Z}_\downarrow (w, (w+\lambda ) p^{-1}; b)
	&=
	\frac{\frac{\lambda}{w+\lambda}
	\sum_{k=1}^\infty
	\left(\frac{
	\mathrm{e}^{b  (w+\lambda)(1-p^{-1}-p^{-k})}-\mathrm{e}^{- b (w+\lambda) p^{-k-1}}}{(w+\lambda)(1 - p^{-k})}
	+\frac{\mathrm{e}^{-b (w+\lambda) p^{-k}}-\mathrm{e}^{-b  (w+\lambda) p^{-k-1}}}{(w+\lambda)p^{-k}}\right)
	\prod_{i=1}^{k-1}\frac{\lambda}{p (w+\lambda)(1- p^{-i})}
	}
	{\sum_{k=0}^\infty  \mathrm{e}^{b(w+\lambda)(1- p^{-k}) }\prod_{i=1}^{k}\frac{\lambda}{p (w+\lambda)(1-p^{-i})}
	}\nonumber\\
	&=\frac{p}{w+\lambda}\mathrm{e}^{-b(w+\lambda)p^{-1}}
	-
	\frac{p}{w+\lambda}\mathrm{e}^{-b(w+\lambda)p^{-1}}
	\frac{1
	}
	{\sum_{k=0}^\infty  \mathrm{e}^{-b(w+\lambda)p^{-k} }\prod_{i=1}^{k}\frac{\lambda}{p (w+\lambda)(1-p^{-i})}
	}\nonumber\\
	&\qquad
	+\frac{\lambda}{w+\lambda}\mathrm{e}^{-b(w+\lambda)p^{-1}}
	\Bigg(
	\frac{
	\sum_{k=1}^\infty
	\frac{\mathrm{e}^{-b (w+\lambda) p^{-k}}}{(w+\lambda)p^{-k}}
	\prod_{i=1}^{k-1}\frac{\lambda}{p (w+\lambda)(1- p^{-i})}
	}
	{\sum_{k=0}^\infty  \mathrm{e}^{-b(w+\lambda)p^{-k} }\prod_{i=1}^{k}\frac{\lambda}{p (w+\lambda)(1-p^{-i})}
	}\nonumber\\
	&\qquad
	-
	\frac{
	\sum_{k=1}^\infty
	\left(\frac{
	\mathrm{e}^{- b (w+\lambda) p^{-k-1}}}{(w+\lambda)(1 - p^{-k})}
	+\frac{\mathrm{e}^{-b  (w+\lambda) p^{-k-1}}}{(w+\lambda)p^{-k}}\right)
	\prod_{i=1}^{k-1}\frac{\lambda}{p (w+\lambda)(1- p^{-i})}
	}
	{\sum_{k=0}^\infty  \mathrm{e}^{-b(w+\lambda)p^{-k} }\prod_{i=1}^{k}\frac{\lambda}{p (w+\lambda)(1-p^{-i})}
	}
	\Bigg)\nonumber\\
	&=\frac{p}{w+\lambda}\mathrm{e}^{-b(w+\lambda)p^{-1}}
	-\frac{w}{w+\lambda}\mathrm{e}^{-b(w+\lambda) }
	\frac{
	\sum_{k=1}^\infty
	\frac{\mathrm{e}^{-b (w+\lambda) p^{-k}}}{(w+\lambda)p^{-k}}
	\prod_{i=1}^{k-1}\frac{\lambda}{p (w+\lambda)(1- p^{-i})}
	}
	{\sum_{k=0}^\infty  \mathrm{e}^{-b(w+\lambda)p^{-k} }\prod_{i=1}^{k}\frac{\lambda}{p (w+\lambda)(1-p^{-i})}
	}
	.\label{Eq:Constant_LS_Sol}
\end{align}
In light of \eqref{Eq:Constant_LS_Sol}, Equation \eqref{Eq:Predownone} yields
	\begin{align}
	\label{downone}
	\tilde{Z}_\downarrow (w, z; b)
	&=\frac{\lambda}{w+\lambda}
	\sum_{k=0}^\infty
	\left(-\frac{\mathrm{e}^{- b z p^{-k-1}}}{w+\lambda -  z p^{-k}}
	+\frac{\mathrm{e}^{-b z p^{-k}}-\mathrm{e}^{- b  z p^{-k-1}}}{z p^{-k}}\right)
	\prod_{i=0}^{k-1}\frac{\lambda}{p (w+\lambda -  z p^{-i})}\\
	&\qquad +\frac{w}{w+\lambda}
	\frac{
	\sum_{l=1}^\infty
	\frac{\mathrm{e}^{-b (w+\lambda) p^{-l}}}{(w+\lambda)p^{-l}}
	\prod_{i=1}^{l-1}\frac{\lambda}{p (w+\lambda)(1- p^{-i})}
	}
	{\sum_{l=0}^\infty  \mathrm{e}^{-b(w+\lambda)p^{-l} }\prod_{i=1}^{l}\frac{\lambda}{p (w+\lambda)(1-p^{-i})}
	}
	\sum_{k=0}^\infty  \mathrm{e}^{- bz p^{-k} }
	\prod_{i=0}^{k}\frac{\lambda}{p (w+\lambda -  z p^{-i})}\nonumber\\
	%
%	&=\frac{\mathrm{e}^{-bz}}{z}
%	+
%	\frac{q}{w+\lambda}
%	\sum_{k=0}^\infty
%	\mathrm{e}^{- bz p^{-k} }
%	\left(
%	\frac{1}{zp^{-k}}+
%	%
%	\frac{
%	\sum_{l=1}^\infty
%	\frac{\mathrm{e}^{-b (w+\lambda) p^{-l}}}{(w+\lambda)p^{-l}}
%	\prod_{i=1}^{l-1}\frac{\lambda}{p (w+\lambda)(1- p^{-i})}
%	}
%	{\sum_{l=0}^\infty  \mathrm{e}^{-b(w+\lambda)p^{-l} }\prod_{i=1}^{l}\frac{\lambda}{p (w+\lambda)(1-p^{-i})}
%	}
%	%
%	\frac{\lambda}{p(w+\lambda-zp^{-k})}
%	\right)\nonumber\\
%	&\qquad	\times \prod_{i=0}^{k-1}\frac{\lambda}{p (w+\lambda -  z p^{-i})}\nonumber\\
	&=\frac{\mathrm{e}^{-bz}}{z}
	-
	\frac{w}{w+\lambda}
	\sum_{k=0}^\infty
	\frac{\mathrm{e}^{- bz p^{-k} }}{zp^{-k}}
	 \prod_{i=0}^{k-1}\frac{\lambda}{p (w+\lambda -  z p^{-i})}
	 +
	 \frac{w}{w+\lambda}\frac{p}{\lambda}
	\tilde{C}(w;b)
	\sum_{k=0}^\infty
	\mathrm{e}^{- bz p^{-k} }
	 \prod_{i=0}^{k}\frac{\lambda}{p (w+\lambda -  z p^{-i})},
	 \nonumber
	\end{align}
with $\tilde{C}(w;b)$	 given in \eqref{Eq:ConstantCLSTDown}.

We proceed now with the inversion of the LST with respect to $z$. To this end, we rewrite \eqref{downone} by expanding the products into summations as follows
	\begin{align*}
	\tilde{Z}_\downarrow(w, z;b)
	&=\frac{\mathrm{e}^{-bz}}{z}
	-
	\frac{w}{w+\lambda}
	\sum_{k=0}^\infty
	\frac{\mathrm{e}^{- bz p^{-k} }\lambda^k}{z}
	 \prod_{i=0}^{k-1}\frac{1}{w+\lambda -  z p^{-i}}\\
	 &\quad+
	 \frac{w}{w+\lambda}\frac{p}{\lambda}
	\tilde{C}(w;b)
	\sum_{k=0}^\infty
	\frac{\mathrm{e}^{- bz p^{-k} }\lambda^{k+1}}{p^{k+1}}
	 \prod_{i=0}^{k}\frac{1}{w+\lambda -  z p^{-i}}\\
	 &=\frac{\lambda}{w+\lambda}\frac{\mathrm{e}^{-bz}}{z}
	+
	\frac{w}{w+\lambda}
	\sum_{k=1}^\infty
	\left(\frac{\lambda}{w+\lambda}\right)^k
	 \left(
	 \frac{1}{(w+\lambda)}
	 \sum_{i=0}^{k-1}
	 \frac{1}{ \prod\limits_{j=0,j\neq i}^{k-1}(1-p^{i-j})}
	 \frac{\mathrm{e}^{- bz p^{-k} }}{z-(w+\lambda)p^{i}}
	 - \frac{\mathrm{e}^{- bz p^{-k} }}{z}
	 \right)\\
	 &\quad-
	 \frac{w}{w+\lambda}\frac{p}{\lambda}
	\tilde{C}(w;b)
	\sum_{k=0}^\infty
	\left(\frac{\lambda}{ p(w+\lambda)}\right)^{k+1}
	 \sum_{i=0}^{k}
	  \frac{p^i}{ \prod\limits_{j=0,j\neq i}^{k}(1-p^{i-j})}
	 \frac{\mathrm{e}^{- bz p^{-k} }}{z-(w+\lambda) p^{i}}\\
%	  &=\frac{\lambda}{w+\lambda}\frac{\mathrm{e}^{-bz}}{z}-\frac{w}{w+\lambda}
%	\sum_{k=1}^\infty
%	\left(\frac{\lambda}{w+\lambda}\right)^k
%	 \frac{\mathrm{e}^{- bz p^{-k} }}{z}
%	+
%	\frac{w}{(w+\lambda)^2}
%	\sum_{k=1}^\infty
%	\left(\frac{\lambda}{w+\lambda}\right)^k
%	 \sum_{i=0}^{k-1}
%	 \frac{1}{ \prod\limits_{j=0,j\neq i}^{k-1}(1-p^{i-j})}
%	 \frac{\mathrm{e}^{- bz p^{-k} }}{z-(w+\lambda)p^{i}}
%	 \\
%	 &\quad-
%	 \frac{w}{w+\lambda}
%	\tilde{C}(w;b)
%	%
%	\sum_{k=1}^\infty
%	\left(\frac{\lambda}{ w+\lambda}\right)^{k}
%	 \sum_{i=0}^{k-1}
%	  \frac{p^{i-k}}{ \prod\limits_{j=0,j\neq i}^{k-1}(1-p^{i-j})}
%	 \frac{\mathrm{e}^{- bz p^{-(k-1)} }}{z-(w+\lambda) p^{i}}\\
	 &=\frac{w}{w+\lambda}
	\sum_{k=1}^\infty
	\left(\frac{\lambda}{w+\lambda}\right)^k
	 \frac{\mathrm{e}^{-bz}-\mathrm{e}^{- bz p^{-k} }}{z}
	+
	\frac{w}{(w+\lambda)^2}
	\sum_{k=1}^\infty
	\left(\frac{\lambda}{w+\lambda}\right)^k
	 \sum_{i=0}^{k-1}
	 \frac{1}{ \prod\limits_{j=0,j\neq i}^{k-1}(1-p^{i-j})}
	 \frac{\mathrm{e}^{- bz p^{-k} }}{z-(w+\lambda)p^{i}}
	 \\
	 &\quad-
	 \frac{w}{w+\lambda}\frac{p}{\lambda}
	\tilde{C}(w;b)
	\sum_{k=0}^\infty
	\left(\frac{\lambda}{ w+\lambda}\right)^{k+1}
	 \sum_{i=0}^{k}
	  \frac{p^{i-k-1}}{ \prod\limits_{j=0,j\neq i}^{k}(1-p^{i-j})}
	 \frac{\mathrm{e}^{- bz p^{-k} }}{z-(w+\lambda) p^{i}}\\
	 &=\frac{w}{w+\lambda}
	\sum_{k=1}^\infty
	\left(\frac{\lambda}{w+\lambda}\right)^k
	\int_{b}^{\infty} \mathrm{e}^{-zx}\mathbb{1}_{\{b<x\leq bp^{-k}\}}\mathrm{d}x\\
	&\quad+
	\frac{w}{\lambda(w+\lambda)}
	\sum_{k=0}^\infty
	\left(\frac{\lambda}{w+\lambda}\right)^{k+1}
	 \sum_{i=0}^{k}
	 \frac{1-p^{i-k}}{ \prod\limits_{j=0,j\neq i}^{k}(1-p^{i-j})}
	\int_{b}^{\infty} \mathrm{e}^{-zx}\mathbb{1}_{\{bp^{-k}<x\}}\mathrm{e}^{(w+\lambda)p^i x}\mathrm{d}x
	 \\
	 &\quad-
	 \frac{w}{\lambda(w+\lambda)}
	 \tilde{C}(w;b)
	\sum_{k=0}^\infty
	\left(\frac{\lambda}{ w+\lambda}\right)^{k+1}
	 \sum_{i=0}^{k}
	  \frac{p^{i-k}}{ \prod\limits_{j=0,j\neq i}^{k}(1-p^{i-j})}
	\int_{b}^{\infty} \mathrm{e}^{-zx}\mathbb{1}_{\{bp^{-k}<x\}}\mathrm{e}^{(w+\lambda)p^i x}\mathrm{d}x\\
	&=\frac{w}{w+\lambda}
	\sum_{k=1}^\infty
	\left(\frac{\lambda}{w+\lambda}\right)^k
	\int_{b}^{\infty} \mathrm{e}^{-zx}\mathbb{1}_{\{b<x\leq bp^{-k}\}}\mathrm{d}x\\
	&\quad+
	\frac{w}{\lambda(w+\lambda)}
	\sum_{k=0}^\infty
	\left(\frac{\lambda}{w+\lambda}\right)^{k+1}
	 \sum_{i=0}^{k}
	 \frac{
	 1-(1+\tilde{C}(w;b))p^{i-k}
	 }{ \prod\limits_{j=0,j\neq i}^{k}(1-p^{i-j})}
	\int_{b}^{\infty} \mathrm{e}^{-zx}\mathbb{1}_{\{bp^{-k}<x\}}\mathrm{e}^{(w+\lambda)p^i x}\mathrm{d}x
	 ,
	\end{align*}
which completes the proof of the theorem.	
\end{proof}

\begin{remark}[Alternative approach]
\label{rem-alt-approach}
{\rm Instead of the above solution, one could equivalently consider that
	\begin{align}
	\tilde{Z}_\downarrow (w, z; b)=\sum_{n=-\infty}^\infty c_n(w;b) z^{n},
	\label{Eq:ZDownLaurent}
	\end{align}
and substitute this into the recursion \eqref{Eq:recursion_LST_DC}. This yields
	\begin{align}
	(w+\lambda-z)\sum_{n=-\infty}^\infty c_n(w;b) z^{n}	
	&=\sum_{n=0}^\infty a_n(w;b) z^{n} +\sum_{n=-\infty}^\infty \lambda c_n(w;b)p^{-n-1} z^{n},
	\label{Eq:RecLaurent}
	%\\
	%(w+\lambda)\sum_{n=-\infty}^\infty c_n(w;b) z^{n}-\sum_{n=-\infty}^\infty c_{n-1}(w;b) z^{n}	
	%&=\sum_{n=0}^\infty a_n(w;b) z^{n} +\sum_{n=-\infty}^\infty \lambda c_n(w;b)p^{-n-1} z^{n}.
	\end{align}
with
	\begin{align*}
	\sum_{n=0}^\infty a_n(w;b) z^{n}	
	&=\frac{\lambda}{w+\lambda }
	\left( \mathrm{e}^{b  (w+\lambda)(1-p^{-1})-bz}-\mathrm{e}^{-bz p^{-1}}+
	\left(\mathrm{e}^{-b z }-\mathrm{e}^{-b z p^{-1}}\right)\frac{w+\lambda-z}{z}\right)\\
	&\quad -\frac{\lambda }{p}\mathrm{e}^{b(w+\lambda- z ) }\tilde{Z}_\downarrow (w, (w+\lambda)p^{-1}; b) .
	\end{align*}
Equating the coefficients of $z^n$, $n\in\mathbb{Z}$, yields
	\begin{align}
	%(w+\lambda)c_n(w;b) - c_{n-1}(w;b) &=a_n(w;b) + \lambda c_n(w;b)p^{-n-1} \\
	c_n(w;b)  &=\frac{a_n(w;b) + c_{n-1}(w;b)}{w+\lambda(1 -p^{-n-1})}\nonumber\\
		       &=\sum_{k=0}^n\frac{a_k(w;b)}{\prod\limits_{i=k+1}^{n+1}(w+\lambda(1 -p^{-i}))}
		       		+\frac{c_{-1}(w;b)}{\prod\limits_{i=1}^{n+1}(w+\lambda(1 -p^{-i}))},
					\qquad n\geq0,\label{Eq:ZDownCoefficientsPositive}\\
	c_{-n}(w;b)  &=\left(w+\lambda(1 -p^{n-2})\right) c_{-n+1}(w;b)\nonumber\\
		          &=\prod\limits_{i=0}^{n-2}\left(w+\lambda(1 -p^{-i})\right) c_{-1}(w;b),
					\qquad n\geq2,\label{Eq:ZDownCoefficientsNegative}
	\end{align}
with $c_{-1}(w;b)=\mathbb{1}_{\{w=0\}}$.
On the one hand, note that $\lim_{n\to\infty}c_{-n}(w;b) =0$. On the other hand, taking the limit in the LHS of Equation \eqref{Eq:ZDownCoefficientsNegative} yields
	\begin{align*}
	\prod\limits_{i=0}^{\infty}\left(w+\lambda(1 -p^{-i})\right) =\begin{cases}
	0,&\text{ if }w=0,\\
	\infty,&\text{otherwise}.
	\end{cases}
	\end{align*}
Moreover, $\tilde{Z}_\downarrow(0, z;b):=\int_{b}^\infty \mathrm{e}^{-z x }\mathrm{d}x=\mathrm{e}^{-bz}/z$.
So all in all, $c_{-1}(w;b)=\mathbb{1}_{\{w=0\}}$ and $c_{-n}(w;b)=0$ for all $n\geq 2$. Note that this is consistent with the result from  Equation \eqref{downone}, as $\lim_{z\to0}z \tilde{Z}_\downarrow (w, z; b)=1-\mathbb{1}_{\{w\neq0\}}=\mathbb{1}_{\{w=0\}}$.

Furthermore, one can compute  $\tilde{Z}_\downarrow (w, (w+\lambda)p^{-1}; b) $ by noting that the sequence $a_n(w;b)$ can be decomposed into
	\begin{align*}
	a_n(w;b)=a_{n,1}(w;b)+\tilde{Z}_\downarrow (w, (w+\lambda)p^{-1}; b) a_{n,2}(w;b),
	\end{align*}
where the sub-sequences $a_{n,1}(w;b)$ and $a_{n,2}(w;b)$ are fully known (they are the coefficients of the Taylor expansions of the exponents). Then, by the definition of \eqref{Eq:ZDownLaurent} and by \eqref{Eq:ZDownCoefficientsPositive}--\eqref{Eq:ZDownCoefficientsNegative},
	\begin{align*}
	\tilde{Z}_\downarrow (w, (w+\lambda)p^{-1}; b) &=\sum_{n=0}^\infty \sum_{k=0}^n
	\frac{\left(a_{k,1}(w;b) +\tilde{Z}_\downarrow (w, (w+\lambda)p^{-1}; b) a_{k,2}(w;b)\right)\left(\frac{w+\lambda}{p}\right)^{n}}
	{\prod\limits_{i=k+1}^{n+1}(w+\lambda(1 -p^{-i}))}\\
	&\quad+ \mathbb{1}_{\{w=0\}}\sum_{n=-1}^\infty
	\frac{\left(\frac{w+\lambda}{p}\right)^{n}}{\prod\limits_{i=1}^{n+1}(w+\lambda(1 -p^{-i}))}
	,
	\end{align*}
which yields after straightforward computations the value of $\tilde{Z}_\downarrow (w, (w+\lambda)p^{-1}; b)$.
%	\begin{align*}
%	\tilde{Z}_\downarrow (w, (w+\lambda)p^{-1}; b) &=\frac{\sum_{n=0}^\infty \sum_{k=0}^n
%	\frac{a_{k,1}(w;b)\left(\frac{w+\lambda}{p}\right)^{n}}
%	{\prod\limits_{i=k+1}^{n+1}(w+\lambda(1 -p^{-i}))}
%	+ \mathbb{1}_{\{w=0\}}
%	\sum_{n=-1}^\infty
%	\frac{\left(\frac{w+\lambda}{p}\right)^{n}}{\prod\limits_{i=1}^{n+1}(w+\lambda(1 -p^{-i}))}
%	}{1-\sum_{n=0}^\infty \sum_{k=0}^n
%	\frac{a_{k,2}(w;b) \left(\frac{w+\lambda}{p}\right)^{n}}
%	{\prod\limits_{i=k+1}^{n+1}(w+\lambda(1 -p^{-i}))}}.
%	\end{align*}
Note that this form of the double LST does not permit a straightforward inversion with respect to $z$ as done in the proof of Theorem \ref{mainthm2}.}
\end{remark}

%\begin{proof}~
%
%\includegraphics[scale=0.5]{log_asymptotics.pdf}
%
%\end{proof}

Having established the laws governing the one-sided exit problems  (up-crossing by creeping or down-crossing by a jump), we proceed with the next two theorems, in which we analyze the double-sided exit problems:

\begin{thm}[Upward two-sided exit problem]
\label{mainthm3}
For $x\in [b,a)$,%b\leq x<a$,
the Laplace-Stieltjes transform
	\begin{equation}
	\label{eq:upward1}
	L_{\uparrow}(w;x,a,b) := \E_x [\mathrm{e}^{-w \tau_{\uparrow a}(x)}\I_{\{\tau_{\uparrow a} (x) < \tau_{\downarrow b}(x)\}}]
	\end{equation}
of the upward two-sided exit time $\tau_{\uparrow a}(x)$ of the additive-increase and multiplicative-decrease process $\left ( X_t\right)_{t\geq 0}$ equals
	\begin{equation}
	\label{eq:upward2}
	L_{\uparrow}(w;x,a,b)=\frac{L_{\uparrow}(w;b,a,b)}{L_{\uparrow}(w;b,x,b)},
	\end{equation}
and $L_{\uparrow}(w;b,x,b)$ (similarly for $L_{\uparrow}(w;b,a,b)$) may be found recursively for all values of $x$ as follows:\\
For $x \in (b,bp^{-1}]$,%\leq b/p$,
	\begin{equation}
	\label{eq:upward3}
	L_{\uparrow}(w;b,x,b) = \mathrm{e}^{-(w+\lambda) (x-b)}.
	\end{equation}
For  $x\in (b/p^{k}, b/p^{k+1}]$ and all  $k\geq 1$,
	\begin{equation}
	\label{eq:upward4}
	L_{\uparrow}(w;b,x,b) = \frac{\mathrm{e}^{-(w+\lambda) (x-b/p^k)}}{\frac{1}{L_{\uparrow}(w;b,b/p^k,b)} -  \int_0^{x-b/p^k} \frac{\lambda \mathrm{e}^{-(\lambda+w ) t}}
	{L_{\uparrow}(w;b,b/p^{k-1} + pt,b)}\mathrm{d}t},
	\end{equation}
where it is assumed that recursively $L_{\uparrow}(w;b,y,b)$ is known for all $y \le b/p^k$, the starting point of the recursion being given by \eqref{eq:upward3}.
\end{thm}
%\RvdH{Is it maybe useful to compute $L_{\uparrow}(w;b,x,b)$ for $b/p^{k} < x \le b/p^{k+1}$ with $k=1, 2$ or even $k=3$? Might we see some structure in such formulas?}

\begin{proof}
First, for all $y\in[b,x)$,
	$$
	\tau_{\uparrow a}(y) \I_{\{\tau_{\uparrow a}(y) < \tau_{\downarrow b}(y)\}}
	\stackrel{d}{=}  \tau_{\uparrow x}(y) \I_{\{\tau_{\uparrow x}(y) < \tau_{\downarrow b}(y)\}} + \tau_{\uparrow a}(x) \I_{\{\tau_{\uparrow a}(x) < \tau_{\downarrow b}(x)\}},
	$$
and note that the random variables on the RHS are independent which follows straightforwardly from the Markov property.
Taking $y=b$  proves \eqref{eq:upward2}. We can therefore focus only on the computation of $L_{\uparrow}(w;b,x,b)$, for $b< x\leq a$, in which the starting position is set to $b$.\\
If $x\in (b, b/p)$, then
	$$
	L_{\uparrow}(w;b,x,b) = \E_b\left[\mathrm{e}^{-w\tau_{\uparrow x}(b)}\I_{\{\tau_{\uparrow a(b)} < \tau_{\downarrow b}(b)\}}\right] = \mathrm{e}^{-w (x-b)} \P(T_1 > x-b) = \mathrm{e}^{-(w+\lambda) (x-b)},
	$$
which proves \eqref{eq:upward3}.\\
Next, assume that  $x\in (b/p^{k}, b/p^{k+1}]$, for  $k\geq 1$. Then, from arguments similar to those used above in the proof of \eqref{eq:upward2},
	\begin{equation} \label{eq:upward5}
	L_{\uparrow}(w;b,x,b) = L_{\uparrow}(w;b,b/p^k,b) L_{\uparrow}(w;b/p^k,x,b).
	\end{equation}
Now note that
	\begin{align*}
	& \tau_{\uparrow x}(b/p^k)\I_{\{\tau_{\uparrow x}(b/p^k) < \tau_{\downarrow b}(b/p^k)\}} \\
	& \qquad \stackrel{d}{=} (x-b/p^k) \I_{\{T > x-b/p^k\}} + \left(T + \tau_{\uparrow x}(p(b/p^k+T))\I_{\{\tau_{\uparrow x}(p(b/p^k+T)) < \tau_{\downarrow b}(p(b/p^k+T))\}}\right)\I_{\{T \le x-b/p^k\}}.
	\end{align*}
Taking Laplace transforms we obtain
	$$
	L_{\uparrow}(w;b/p^k,x,b) = \mathrm{e}^{-(w+\lambda)(x-b/p^k)} + \int_0^{x-b/p^k} \lambda \mathrm{e}^{-(w+\lambda)t} L_{\uparrow}(w;b/p^{k-1}+pt,x,b)\mathrm{d}t.
	$$	
We can now apply \eqref{eq:upward5} to rewrite the above as
	$$
	\frac{L_{\uparrow}(w;b,x,b)}{L_{\uparrow}(w;b,b/p^k,b)}
	= \mathrm{e}^{-(w+\lambda) (x-b/p^k)} + \int_0^{x-b/p^k} \lambda \mathrm{e}^{-(\lambda+w ) t} \frac{L_{\uparrow}(w;b,x,b)}{L_{\uparrow}(w;b,b/p^{k-1} + pt,b)}\mathrm{d}t,
	$$
which leads to \eqref{eq:upward4}.
\end{proof}
\medskip

\begin{remark}[Relation Theorems \ref{mainthm1} and  \ref{mainthm3}]
\label{rem:two_sided_up}
{\rm
The result of Theorem \ref{mainthm1} might be retrieved from Theorem \ref{mainthm3}, when we use the fact that $\lim_{b \to 0} L_{\uparrow}(w;x,a,b) = Z_{\uparrow}(w;x,a)$ when $a$ and $x$ are fixed. In particular, note that \eqref{eq:upward4} looks very similar to, for instance, the equation just above \eqref{eq:thm1_aux3}. However, it is hard to perform this limit from the result in Theorem \ref{mainthm3}, as, when we let $b$ tend to $0$ while leaving $x$ fixed, $b/p^k$ tends to zero as well. This makes it difficult to explicitly compute the limit. Thus, while theoretically the scale function $Z_\uparrow(w;x,a)$ can be directly related to $L_{\uparrow}(w;x,a,b)$, in practice this is difficult.
}
\end{remark}

\begin{remark}[Convenient rewrite of \eqref{eq:upward4}]
\label{rem:upward4-rew}
{\rm
Note also that, if one introduces $K_{\uparrow}(w;a,x,b) = 1/L_{\uparrow}(w;a,x,b)$, then equation \eqref{eq:upward4}, rewritten for $K_{\uparrow}$, for $x\in (b/p^{k}, b/p^{k+1}]$ and all  $k\geq 1$, simplifies to
	\eqan{
	K_{\uparrow}(w;b,x,b) &= \mathrm{e}^{(w+\lambda) (x-b/p^k)} K_{\uparrow}(w;b,b/p^k,b) -  \mathrm{e}^{(w+\lambda) (x-b/p^k)}
	\int_{0}^{x-b/p^k} \frac{\lambda}{p} \mathrm{e}^{-(\lambda+w )t}K_{\uparrow}(w;b,b/p^{k-1}+pt,b)\mathrm{d}t\\
	&=\mathrm{e}^{(w+\lambda) (x-b/p^k)} K_{\uparrow}(w;b,b/p^k,b) -  \mathrm{e}^{(w+\lambda) (x-b/p^k)}
	\int_{b/p^{k-1}}^{px} \frac{\lambda}{p} \mathrm{e}^{-(\lambda+w )(s/p-b/p^k)}K_{\uparrow}(w;b,s,b)\mathrm{d}s\nn\\
	&=\mathrm{e}^{(w+\lambda) (x-b/p^k)} K_{\uparrow}(w;b,b/p^k,b) -  \mathrm{e}^{(w+\lambda)x}
	\int_{b/p^{k-1}}^{px} \frac{\lambda}{p} \mathrm{e}^{-(\lambda+w)s/p}K_{\uparrow}(w;b,s,b)\mathrm{d}s\nn.
	}
where $K_{\uparrow}(w;b,x,b) = \mathrm{e}^{(w+\lambda) (x-b)}$ for $x \in (b,b/p]$.

This also allows us to perform the iteration explicitly for more values of $x$. Indeed, for $k=1$ and thus
$x\in (b/p, b/p^{2}]$, the recursion yields
	\eqan{
	\label{K2-comp}
	K_{\uparrow}(w;b,x,b) &= \mathrm{e}^{(w+\lambda) (x-b/p)} K_{\uparrow}(w;b,b/p,b) - \mathrm{e}^{(w+\lambda)x}
	\int_{b}^{px} \frac{\lambda}{p} \mathrm{e}^{-(\lambda+w)s/p}K_{\uparrow}(w;b,s,b)\mathrm{d}s\\
	&=\mathrm{e}^{(w+\lambda)(x-b/p)}\mathrm{e}^{(w+\lambda) (b/p-b)}-\mathrm{e}^{(w+\lambda) x}
	\int_b^{xp} \frac{\lambda}{p} \mathrm{e}^{-(\lambda+w )s/p} \mathrm{e}^{(w+\lambda) (s-b)}\mathrm{d}s\nn\\
	&=\mathrm{e}^{(w+\lambda)(x-b)}-\mathrm{e}^{(w+\lambda) x}
	\int_b^{xp} \frac{\lambda}{p} \mathrm{e}^{-(w+\lambda)s (1-1/p)}\mathrm{d}s\nn\\
	&=\mathrm{e}^{(w+\lambda)(x-b)}-\mathrm{e}^{(w+\lambda)x}
	\frac{\lambda}{(1-p)(w+\lambda)}\big[\mathrm{e}^{-(w+\lambda)b(1-1/p)}-\mathrm{e}^{-(w+\lambda)x (1-p)}\big].\nn
	}
We extend this one iteration further, to obtain, for $k=2$ and thus $x\in (b/p^{2}, b/p^{3}]$,
	\eqan{
	K_{\uparrow}(w;b,x,b) &= \mathrm{e}^{(w+\lambda) (x-b/p^2)} K_{\uparrow}(w;b,b/p^2,b) -  \mathrm{e}^{(w+\lambda) x}
	\int_{b/p}^{px} \lambda \mathrm{e}^{-(\lambda+w ) s/p}K_{\uparrow}(w;b,s,b)\mathrm{d}s.
	}
After this, we can substitute \eqref{K2-comp} to compute $K_{\uparrow}(w;b,x,b)$ for $x\in (b/p^{2}, b/p^{3}]$. By iteration, it is easily seen that there exist coefficients $a_{n,k}=a_{n,k}(p,w,\lambda)$ such that, for $x\in (b/p^{k}, b/p^{k+1}]$ and all $k\geq 1$,
	\eqn{
	\label{Kuparrow-repr}
	K_{\uparrow}(w;b,x,b)=\sum_{n=0}^k a_{n,k} \mathrm{e}^{(w+\lambda)p^n x}.
	}
In turn, this representation looks similar to \eqref{zdownarrowformula}, but here we have no exact formula for the coefficients $a_{n,k}$, whereas in \eqref{zdownarrowformula} we do.
}
\end{remark}

\begin{remark}[Continuity and differentiability properties of $a\mapsto L_{\uparrow}(w;x,a,b)$]
\label{rem:con-diff-Luparrow}
{\rm The above formulas are also convenient to study the continuity and differentiability properties of $x\mapsto L_{\uparrow}(w;b,x,b)$. Indeed, by $K_{\uparrow}(w;x,a,b) = 1/L_{\uparrow}(w;x,a,b)$ and \eqref{Kuparrow-repr}, $x\mapsto L_{\uparrow}(w;b,x,b)$ is continuous everywhere, while it is continuously differentiable in every point except possibly in $x=b/p^k$ for all $k\geq 0$. In this countable number of points, however, left- and right-derivatives do exist. Through \eqref{eq:upward2}, this can be extended to continuity and almost everywhere differentiability of $a\mapsto L_{\uparrow}(w;x,a,b)$ for general $x$ and $b$.
}
\end{remark}

\begin{thm}[Downward two-sided exit problem]
\label{mainthm4}
For $x\in [b,a)$, %b\leq x<a$,
the Laplace-Stieltjes transform
	\begin{equation}
	\label{eq:downward1}
	L_{\downarrow}(w;x,a,b)=\E_x\left[\mathrm{e}^{-w\tau_{\downarrow b}(x)}\I_{\{\tau_{\uparrow a}(x)>\tau_{\downarrow b}(x) \}}\right]
	\end{equation}
of the downward two-sided exit time $\tau_{\downarrow b}(x)$ of the additive-increase and multiplicative-decrease process $\left ( X_t\right)_{t\geq 0}$ equals
	\begin{equation}
	\label{downone0}
	L_{\downarrow}(w;x,a,b)
	= Z_\downarrow(w;x,b)-L_{\uparrow}(w;x,a,b)Z_\downarrow(w;a,b).
%\E_x\left[\mathrm{e}^{-w\tau_{\downarrow b}(x)} \right] - \E_x\left[\mathrm{e}^{-w\tau_{\uparrow a}(x)}\I_{\{\tau_{\uparrow a}(x) < %\tau_{\downarrow b}(x)\}} \right]
%	\E\left[\mathrm{e}^{-w \tau_{\downarrow b}(a)} \right] ,
	\end{equation}
%where  $\E\left[\mathrm{e}^{-w \tau_{\downarrow b}(a)} \right]$ is computed in Theorem \ref{mainthm1}, $\E_x\left[\mathrm{e}^{-w\tau_{\downarrow %b}(x)} \right]$ in Theorem \ref{mainthm2}, and $\E_x\left[\mathrm{e}^{-w\tau_{\uparrow a}(x)}\I_{\{\tau_{\uparrow a}(x) < \tau_{\downarrow %b}(x)\}} \right]$ in Theorem \ref{mainthm3}.
\end{thm}

\begin{proof}
Clearly, for all $x\in [b,a)$,%b \le x < a$,
	\begin{align*}
 	\tau_{\downarrow b}(x)\I_{\{\tau_{\uparrow a}(x) < \tau_{\downarrow b}(x)\}} &\stackrel{d}{=} \tau_{\uparrow a}(x)\I_{\{\tau_{\uparrow a}(x) < \tau_{\downarrow b}(x)\}} +\tau_{\downarrow b}(a),
 	\end{align*}
and note again that the random variables on the RHS are independent. Thus,
	\begin{align*}
	\E_x\left[\mathrm{e}^{-w\tau_{\downarrow b}(x)}\I_{\{\tau_{\uparrow a}(x) < \tau_{\downarrow b}(x)\}}
 	\right]&=\E_x\left[\mathrm{e}^{-w\tau_{\uparrow a}(x)}\I_{\{\tau_{\uparrow a}(x) < \tau_{\downarrow b}(x)\}} \right]\E_x\left[\mathrm{e}^{-w\tau_{\downarrow b}(a)} \right].
 	\end{align*}
Noting that $\E_x\left[\mathrm{e}^{-w\tau_{\downarrow b}(x)}\right]=\E_x\left[\mathrm{e}^{-w\tau_{\downarrow b}(x)}\I_{\{\tau_{\uparrow a}(x) < \tau_{\downarrow b}(x)\}}\right]+\E_x\left[\mathrm{e}^{-w\tau_{\downarrow b}(x)}\I_{\{\tau_{\uparrow a}(x) > \tau_{\downarrow b}(x)\}}\right]$ and using Theorems \ref{mainthm2} and \ref{mainthm3} completes the proof of the theorem.
\end{proof}

\begin{remark}[Relation Theorems \ref{mainthm2} and  \ref{mainthm4} and number of scale functions]
\label{rem:two_sided_down}
{\rm
Note that $\tau_{\downarrow b}(a) \ge T_1 + \cdots + T_{\log_{1/p} (a/b)}$ for independent $T_1, T_2,\ldots\sim {\sf Exp}(\lambda)$ and hence $\tau_{\downarrow b}(a) \to \infty$ a.s. as $a \to \infty$ while $b$ remains fixed. As $L_{\uparrow}(w;x,a,b) \le 1$ for any set of parameters, if we let $a \to \infty$ with a fixed $b$ in \eqref{downone0}, we obtain that the second term vanishes, so that
	$$
	\lim_{a \to \infty} L_{\downarrow}(w;x,a,b)=\lim_{a \to \infty}\E_x\left[\mathrm{e}^{-w\tau_{\downarrow b}(x)}\I_{\{\tau_{\uparrow a}(x)>\tau_{\downarrow b}(x) \}}\right] = Z_\downarrow(w;x,b),
	$$
as can be expected from Theorem \ref{mainthm2}. Thus, the scale function $Z_\downarrow(w;x,b)$ can be directly related to $L_{\uparrow}(w;x,a,b)$. Since \eqref{downone0} in Theorem \ref{mainthm4} also identifies $L_{\downarrow}(w;x,a,b)$ in terms of $L_{\uparrow}$ and $Z_\downarrow$, we conclude that in total we need the {\em two} scale functions $Z_{\uparrow}$ and $L_{\uparrow}$, rather than four.
}
\end{remark}

\section{Exit identities for reflected processes}
\label{sec:main2}

We now consider two types of reflected versions of the process $X$:
\medskip

\noindent
$\rhd$ For the first, we reflect at $a$, i.e., the stochastic process grows linearly over time until it reaches $a$, then it stays there until the next jump occurs, and at the jump time, the process jumps from $a$ to $ap$. Let us denote this process by $X^{\overline{a}}=(X_t^{\overline{a}})_{t\geq0}$.
\medskip

\noindent
$\rhd$ For the second, we reflect at $b$, i.e., the process grows linearly and whenever due to a jump the process jumps over the downward level $b$ it is put back to $b$ and it continues its evolution in time according to the background process $X$ from level $b$. Let us denote this process by $X^{\underline{b}}=(X_t^{\underline{b}})_{t\geq0}$.
\medskip

In this section, we will analyze the first passage times of these two reflected processes defined as
	\eqn{
	\tau_{\downarrow c}^{\overline{a}}(x)=\inf\{t\geq 0\colon X_t^{\overline{a}}< c\mid X_0=x\},\qquad \tau_{\uparrow c}^{\underline{b}}(x)=\inf\{t\geq 0\colon X_t^{\underline{b}}> c\mid X_0=x\}.
	}

Finally, in this paper we identify the Laplace transforms of the first passage time
	\eqn{
	\label{tau-c-def}
	\tau_c(x)=\inf\{t\geq 0\colon Y_t> c\mid X_0=x\}
	}
of the process $Y_t=\overline{X}_t-X_t$ reflected at a running supremum $\overline{X}_t=\sup_{s\leq t} X_s \vee \overline{X}_0$, as well as the first passage time
	\eqn{
	\label{widehat-tau-def}
	\widehat{\tau}_c(x)=\inf\{t\geq 0 \colon \widehat{Y}_t> c\mid X_0=x\}
	}
of the process $\widehat{Y}_t=X_t-\underline{X}_t$ reflected at its running infimum $\underline{X}_t=\inf_{s\leq t} X_s \wedge \underline{X}_0$.

Note that the process $Y$ stays at $0$ until the first jump epoch $T$ of $X$. Then, right after the first jump it equals $Y_T=(1-p)X_{T-}$, and hence the jump of $Y$ is positive ($\Delta Y_T= Y_T>0$). Later, $t\mapsto Y_t$ decreases until the next jump.

The process $\widehat{Y}_t$ evolves in a different way. At the beginning, it equals $t-x$ until the first jump. Then, if at the epoch $T$ of the first jump of $X$ we have $X_T\geq x$, then the process $\widehat{Y}$ evolves without any changes (except a shift by the initial position $x$). If $X_T<x,$ instead, then $\widehat{Y}_T =0$. In this case, our new initial position equals $X_T$ and the process $\widehat{Y}$ evolves as before.

%Assuming the linear positive growth between jumps and that $0<X_0=x<z$, note that the process $Y$ decreases linearly until first hitting $z$ by %$X$. At this moment process $Y$ takes the value $0$. it stays there for exponential time waiting for the next jump $T$ of $X$. At this moment of %time it jumps to the positive value $(1-p)\overline{X}_{T}$ and we continue above evolution with new $z =\overline{X}_{T}$ and %$x=p\overline{X}_{T}$. Note that process $Y$ decreases between jumps and all jumps are upward in contrast to the process $X^a$.
%the jumps depend on the present position of $X$ (not reflected process).

%The process $\widehat{Y}$ moves like $X-b_0$ until first downward passage time, say $T$ of $b_0$ when it takes the value $0$.
%Then it starts evolution like before with new $b_0=X_{T}$, that is, in particular it starts from linear growth.
%Hence the process $\widehat{Y}$ grows linearly between jumps and it jumps downward to $0$ whenever $X$ crosses its past minimum.

We will now present results concerning the exit problems for the reflected processes:

\begin{thm}[First passage time for the reflected process at the upper level]
\label{mainthm5}
For  $x\in [c,a)$, %c \le x < a$,
the Laplace-Stieltjes transform of the first passage time $\tau_{\downarrow c}^{\bar a}(x)$ for the reflected process at the upper level is given by
	\begin{equation}
	\label{refldown}
	\E_x[\mathrm{e}^{-w\tau_{\downarrow c}^{\bar a}(x)}]=L_{\downarrow}(w;x,a,c)
	+L_{\uparrow}(w;x,a,c)\frac{\lambda}{w+\lambda } \left(\frac{L_{\downarrow}(w;pa,a,b)(w+\lambda)}{w+\lambda-L_{\uparrow}(w;pa,a,b)	\lambda}\mathbb{1}_{\{pa>c\}}+\mathbb{1}_{\{pa\leq c\}}\right).
	\end{equation}
%for
%\begin{eqnarray*}
%A(x)&=&Z^{(w)}_c(x)-\frac{W_{c,a}^{(w)}(x)}{W_{c,a}^{(w)}(a)}Z^{(w)}_c(a),\\
%B(x)&=&\frac{W_{c,a}^{(w)}(x)}{W_{c,a}^{(w)}(a)}
%\end{eqnarray*}
\end{thm}

\begin{proof} To prove \eqref{refldown} note that, by the Markov property, for $a>c/p$,
	\eqn{
	\label{refldown-a}
	\E_x[\mathrm{e}^{-w\tau_{\downarrow c}^{\bar a} (x)}]=\E_x[\mathrm{e}^{-w\tau_{\downarrow c}(x)}\mathbb{1}_{\{\tau_{\downarrow c}(x)<\tau_{\uparrow a}(x)\}}]
	+\E_x[\mathrm{e}^{-w\tau_{\uparrow a}(x)}\mathbb{1}_{\{\tau_{\uparrow a}(x)<\tau_{\downarrow c}(x)\}}]\frac{\lambda}{w+\lambda}\E_{pa}[\mathrm{e}^{-w\tau_{\downarrow c}^{\bar a}(pa)}].
	}
Using that \eqref{eq:downward1} implies that
	\[
	\E_x[\mathrm{e}^{-w\tau_{\downarrow c}(x)}\mathbb{1}_{\{\tau_{\downarrow c}(x)<\tau_{\uparrow a}(x)\}}]=L_{\downarrow}(w;x,a,c),
	\]
the first term in \eqref{refldown-a} is equal to the first part of \eqref{refldown}.
%\RvdH{I do not understand this. What is the first part, and how is it related to the above equality?}

To investigate the second term in \eqref{refldown}, we start by noting that, by \eqref{eq:upward1},
	\eqn{
	\E_x[\mathrm{e}^{-w\tau_{\uparrow a}(x)}\mathbb{1}_{\{\tau_{\uparrow a}(x)<\tau_{\downarrow c}(x)\}}]=L_{\uparrow}(w;x,a,c),
	}
so that \eqref{refldown-a} becomes
	\eqn{
	\label{refldown-b}
	\E_x[\mathrm{e}^{-w\tau_{\downarrow c}^{\bar a} (x)}]=L_{\downarrow}(w;x,a,c)
	+L_{\uparrow}(w;x,a,c)\frac{\lambda}{w+\lambda}\E_{pa}[\mathrm{e}^{-w\tau_{\downarrow c}^{\bar a}(pa)}].
	}
Taking $x=pa$, we obtain a linear equation for $\E_{pa}[\mathrm{e}^{-w\tau_{\downarrow c}^{\bar a}(pa)}]$ that can be solved as
	\eqn{
	\label{refldown-b}
	\E_{pa}[\mathrm{e}^{-w\tau_{\downarrow c}^{\bar a}(pa)}]=\frac{L_{\downarrow}(w;x,a,c)}{1-L_{\uparrow}(w;x,a,c)\frac{\lambda}{w+\lambda}}.
	}
This proves \eqref{refldown} for $a>c/p$.

For $a\leq c/p$, instead, we note that $\E_{pa}[\mathrm{e}^{-w\tau_{\downarrow c}^{\bar a}(pa)}]=1$, so that now \eqref{refldown-a} becomes
	\eqn{
	\E_x[\mathrm{e}^{-w\tau_{\downarrow c}^{\bar a}(x)}]
	=\E_x[\mathrm{e}^{-w\tau_{\downarrow c}(x)}\mathbb{1}_{\{\tau_{\downarrow c}(x)<\tau_{\uparrow a}(x)\}}]
	+\E_x[\mathrm{e}^{-w\tau_{\uparrow a}(x)}\mathbb{1}_{\{\tau_{\uparrow a}(x)<\tau_{\downarrow c}(x)\}}]\frac{\lambda}{w+\lambda}.
	}
Again using how the expectations can be translated into $L_{\downarrow}(w;x,a,c)$ and $L_{\uparrow}(w;x,a,c)$, this completes the proof.
\end{proof}

\begin{thm}[First passage time for the reflected process at the lower level]
\label{mainthm5b}
For  $x\in [b,c)$,%b \le x < c$,
the Laplace-Stieltjes transform of the first passage time $\tau_{\uparrow c}^{\underline{b}}(x)$ for the reflected process at the lower level is given by
	\begin{equation}
	\label{reflup}
	\E_x[\mathrm{e}^{-w\tau_{\uparrow c}^{\underline{b}}(x)}]=L_{\uparrow}(w;x,b,c)+L_{\downarrow}(w;x,b,c)\frac{L_{\uparrow}(w;b,b,c)}{1-L_{\downarrow}(w;b,b,c)}.
	\end{equation}
\end{thm}

\begin{proof}
The proof of the theorem is similar to the proof of Theorem \ref{mainthm5}, and our exposition is brief. Indeed, starting from $x$ either
we go to level $c$ before visiting $b$ or the other way around. In the latter case we start from level $b$.
Hence, \eqref{refldown-a} now becomes
	\eqn{
	\label{reflup-b}
	\E_x[\mathrm{e}^{-w\tau_{\uparrow c}^{\underline{b}}(x)}]=\E_x[\mathrm{e}^{-w\tau_{\uparrow c}(x)}\mathbb{1}_{\{\tau_{\uparrow c}(x)<\tau_{\downarrow b}(x)\}}]
	+\E_x[\mathrm{e}^{-w\tau_{\downarrow b}(x)}\mathbb{1}_{\{\tau_{\downarrow b}(x)<\tau_{\uparrow c}(x)\}}]\E_b[\mathrm{e}^{-w\tau_{\uparrow c}^{\underline{b}}(b)}],
	}
which in turn can be written as
	\eqn{
	\label{reflup-b}
	\E_x[\mathrm{e}^{-w\tau_{\uparrow c}^{\underline{b}}(x)}]=L_{\uparrow}(w;x,b,c)+L_{\downarrow}(w;x,b,c)\E_b[\mathrm{e}^{-w\tau_{\uparrow c}^{\underline{b}}(b)}].
	}
Now taking $x=b$ in the above and using it to calculate $\E_b[\mathrm{e}^{-w\tau_{\uparrow c}^{\underline{b}}(b)}]$ completes the proof.
\end{proof}
\medskip

The statements for the process $X$ reflected at the running supremum and infimum are much more complex. They are much more important though as they describe behaviour of so-called drawdown and drawup
processes $Y$ and $\widehat{Y}$. We start with the exit times when the process is reflected at the supremum. In the statement, we write $\partial L_{\uparrow}(w;z,y,u)=(\partial_+ L_{\uparrow}(w;z,v,u)/\partial v)_{v=y}$ for partial right derivative $\partial_+$
and similarly for $L_{\downarrow}$, where these derivatives exist due to Remark \ref{rem:con-diff-Luparrow}, identity \eqref{downone0}
and definition of the function $Z_\downarrow(w;x,b)$ given in \eqref{zdownarrowformula} and are continuous except
countable many points:

\begin{thm}[First passage time for the reflected process at running supremum]
\label{mainthm6}
Assume that $\overline{X}_0=X_0=x$. Then,
the Laplace-Stieltjes transform of the exit time $\tau_c(x)$ of the process $Y$ that is the reflected version of $X$ reflected at the running supremum, equals
	\begin{align}
	\E_x[\mathrm{e}^{-w \tau_c(x)}]
	\label{reflsupremum}
	&=\int_x^{+\infty}\frac{\partial L_{\uparrow}(0;w,w,w-c)}{L_{\uparrow}(0;w,w,w-c)}
	\exp\left\{-\int_x^w \frac{\partial L_{\uparrow}(w;z,z,z-c)}{L_{\uparrow}(w;z,z,z-c)}
	\;\mathrm{d}z\right\}%\nonumber\\&\quad \cdot
	\frac{ \partial L_{\downarrow}(w;w,w,w-c)}{\partial L_{\downarrow}(0;w,w,w-c)}
	\; \mathrm{d}w.
	\end{align}
\end{thm}
\medskip

\begin{remark}[More general initial positions]
\label{rem-gen-init-cond}
\rm In the above theorem, one may consider a more general initial position of the reflected processes
than zero. For example, to get \eqref{reflsupremum} for $\overline{X}_0=z\geq x =X_0$, one can instead consider
	\begin{eqnarray*}\lefteqn{
	\E_x[\mathrm{e}^{-w \tau_c(x)} \mid X_0=x, \overline{X}_0=z]}\\&&
	=\E_x[\mathrm{e}^{-w \tau_{\downarrow z-c}(x)}\mathbb{1}_{\{\tau_{\downarrow z-c}(x)<\tau_{\uparrow z}(x)\}}]
	+\E_x[\mathrm{e}^{-w\tau_{\uparrow z}(x)}\mathbb{1}_{\{\tau_{\uparrow z}(x)<\tau_{\downarrow z-c}(x)\}}]
    	\E_z[\mathrm{e}^{-w \tau_c(z)} \mid X_0=z, \overline{X}_0=z]
	\end{eqnarray*}
and apply Theorems \ref{mainthm3} and \ref{mainthm4}.
\end{remark}

\begin{remark}[Alternative approach to Theorem \ref{mainthm6}]
\label{rem-alt-approach-thm6}
\rm
The first passage time for the reflected process at running maximum considered in Theorem \ref{mainthm6} could be analyzed using \cite[Thm. 3.1 and Ex. 3.5]{Davidetal} in terms of solution of some integral equation. We decided to do it in a more explicit way.
\end{remark}

%
%\begin{remark}\rm
%Identifying the law of
%the first passage time for the reflected process at running infimum, that is the process $X_t-\inf_{s\leq t} X_s$
%is still open.
%Note that
%\[u+c-x\leq \widehat{\tau}_c\leq \sum_{i=1}^{N_c}T_i\]
%where $N_c=\min\{k: T_k\geq c\}$ and thus
%\[\P(\widehat{\tau}_c>t)\leq \sum_{n=1}^\infty \int_c^\infty \lambda \mathrm{e}^{-\lambda s} \P\left(s+\sum_{i=1}^{n-1}T_i>t, \max_{i\leq n-1}T_i<c\right)\; ds\]
%where joint distribution of $(\sum_{i=1}^n T_i, \max_{i\leq n} T_i)$ for $T_i$ being i.i.d random variables with exponential distribution
%is identified in \cite{Panorska}.
%\end{remark}
%
\begin{proof}[Proof of Theorem \ref{mainthm6}]
To prove \eqref{reflsupremum}, we adapt the argument of \cite{Lehoczky} executed for diffusion process. In the first step, we will find the law of $\overline{X}_{\tau_c(x)}$. To find $\P_x(\overline{X}_{\tau_c(x)}>w)$ we partition the interval $[x,w]$ into $n$ subintervals $[s_{i}^n, s_{i+1}^n]$
($i=0,1,\ldots, n-1$)
such that $m_n=\max_i (s_{i+1}^n-s_i^n)\to 0$ as $n\rightarrow +\infty$.
We approximate
$\P_x(\overline{X}_{\tau_c(x)}>w)$ by
$\P(A_n)$ for
	\[
	A_n=\bigcap_{i=0}^{n-1} \{X \text{~hits~} s_{i+1}^n \text{~before~}X\text{~jumps below~}s^n_{i}-c\}.
	\]
To do this, we have to prove that this approximation does not depend on the chosen partition, for which we use the fact that the process $X$ crosses new levels upward in a continuous way, so that $\P_x(\overline{X}_{\tau_c(x)}>w)=\lim_{n\to+\infty}\P(A_n)$.

Then, by the Markov property and Theorem \ref{mainthm3},
	\begin{eqnarray}
	\P_x(\overline{X}_{\tau_c(x)}>w)=\lim_{n\to+\infty}\P(A_n)&=&\exp\left\{\lim_{n\to+\infty} \sum_{i=0}^{n-1} \log
	L_{\uparrow}(0;s^n_i,s^n_{i+1},s^n_i-c)\right\}\nonumber\\&=&
	\exp\left\{-\lim_{n\to+\infty} \sum_{i=0}^{n-1} (s^n_{i+1}-s^n_{i})\frac{1}{s^n_{i+1}-s^n_{i}}\log\left(1- \frac{L_{\uparrow}(0;s^n_i,s^n_{i+1},s^n_i-c)-1}{L_{\uparrow}(0;s^n_i,s^n_{i+1},s^n_i-c)}\right)\right\}
	\nonumber\\
	&=&\exp\left\{-\int_x^w \frac{\partial L_{\uparrow}(0;z,z,z-c)}{L_{\uparrow}(0;z,z,z-c)}\;\mathrm{d}z
	\right\},
	\label{passing}\end{eqnarray}
where we have used the fact that $\partial L_{\uparrow}(0;z,z,z-c)$ is a continuous function of $z$ except possible in countable many points, so that the above Riemann integral is well defined.
As a result,
	\eqn{
	\label{dens-overline-X-tau-c}
	\P(\overline{X}_{\tau_c(x)}\in dw)=\exp\left\{-\int_x^w \frac{\partial L_{\uparrow}(0;z,z,z-c)}{L_{\uparrow}(0;z,z,z-c)}\;\mathrm{d}z\right\} \frac{\partial L_{\uparrow}(0;w,w,w-c)}{L_{\uparrow}(0;w,w,w-c)}dw.
	}

Define the sequence of stopping times
	\[
	\varrho^n_k=\inf\{t\geq 0\colon X_{\Xi^n_{k-1}+t}-X_{\Xi^n_{k-1}}=s^n_k-s^n_{k-1} \quad \text{or} \quad -c\},
	\]
for $\Xi_k^n=\sum_{i=1}^k\varrho^n_i$.
Then
	\[
	\E_x \left[ \mathrm{e}^{-w \tau_c(x)} \mid \overline{X}_{\tau_c(x)}=w\right]
	=\lim_{n\to+\infty}\prod_{i=1}^n\E_x\left[ \mathrm{e}^{-w \varrho^n_i} \mid X_{\Xi^n_{i-1}}=s^n_{i-1}, X_{\Xi^n_{i}}=s^n_{i}\right]
	\E_x \left[ \mathrm{e}^{-w \varrho^n_{n+1}} \mid X_{\Xi^n_{n}}=w, X_{\Xi^n_{n+1}}\leq w-c \right],
	\]
with $s_{n+1}^n>w$ such that $s^n_{n+1}-w$ tends to $0$ as $n\rightarrow+\infty$.
By Theorems \ref{mainthm3} and \ref{mainthm4},
	\[
	\E_x\left[ \mathrm{e}^{-w \varrho^n_i}\mid X_{\Xi^n_{i-1}}=s^n_{i-1}, X_{\Xi^n_{i}}=s^n_{i}\right]=\frac{L_{\uparrow}(w;s^n_{i-1},s^n_{i},s^n_{i-1}-c)}{L_{\uparrow}(0;s^n_{i-1},s^n_{i},s^n_{i-1}-c)},
	\]
and
	\[
	\E_x \left[ \mathrm{e}^{-w \varrho^n_{n+1}} \mid X_{\Xi^n_{n}}=w, X_{\Xi^n_{n+1}}\leq w-c \right]
	=\frac{L_{\downarrow}(w;w,s^n_{n+1},w-c)}{L_{\downarrow}(0;w,s^n_{n+1},w-c)}
	=\frac{L_{\downarrow}(w;w,s^n_{n+1},w-c)/(s_{n+1}^n-w)}{
	L_{\downarrow}(0;w,s^n_{n+1},w-c)/(s_{n+1}^n-w)}.
	\]
Hence applying the same limiting arguments as in \eqref{passing}, we derive
	\begin{align}
	\label{ingred-2-mail6}
	\E_x \left[ \mathrm{e}^{-w \tau_c(x)} \mid \overline{X}_{\tau_c(x)}=w\right]
	&=\exp\left\{-\int_x^w \frac{\partial L_{\uparrow}(w;z,z,z-c)}{L_{\uparrow}(w;z,z,z-c)}\;dz\right\}
	\exp\left\{\int_x^w \frac{\partial L_{\uparrow}(0;z,z,z-c}{L_{\uparrow}(0;z,z,z-c)}\;dz\right\}\\
	&\qquad\times \frac{(\partial L_{\downarrow}(w;w,w,w-c)}{\partial L_{\downarrow}(0;w,w,w-c)}.\nn
	\end{align}
Using the fact that $\E_x \left[ \mathrm{e}^{-w \tau_c(x)}\right]=\int_x^{+\infty} \E_x \left[ \mathrm{e}^{-w \tau_c(x)}|\overline{X}_{\tau_a(x)}=w\right]\P(\overline{X}_{\tau_c(x)}\in dw)$, together with \eqref{dens-overline-X-tau-c} and \eqref{ingred-2-mail6}, completes the proof.
\end{proof}
\medskip

To analyze reflection at the running infimum we will use martingale theory for the first time. We first set the stage. Let $a(w,c,u)$ be a solution of the equation
	\begin{equation}
	\label{acb}
	Z_\downarrow(w;u+c,u)+L_{\uparrow}(w;u+c,a(w,c,u),u)=1.
	\end{equation}
We first argue that the solution to \eqref{acb} always exists.  Note that above equation is equivalent to
	\[
	\E_{u+c} [\mathrm{e}^{-w \tau_{\downarrow u}}]+\E_{u+c} [\mathrm{e}^{-w \tau_{\uparrow a(w,c,u)}(x)}\I_{\{\tau_{\uparrow a(w,c,u)} (x) < \tau_{\downarrow u}(x)\}}]
	=1.
	\]
By Theorem \ref{mainthm3}, $a\mapsto L_{\uparrow}(w;u+c,a,u)$ is continuous.
Moreover,
	\[
	L_{\uparrow}(w;u+c,a,u)=\E_{u+c} [\mathrm{e}^{-w \tau_{\uparrow a}(x)}\I_{\{\tau_{\uparrow a} (x) < \tau_{\downarrow u}(x)\}}]
	\]
tends to $1$ as $a\downarrow u+c$, and to $0$ as $a\uparrow+\infty$.  Since $Z_\downarrow(w;u+c,u)<1$, the solution of \eqref{acb} indeed always exists.
%We start from observation that the function $Z_{\downarrow}(w;x,0)$ treated as a decreasing limit
%of $Z_\downarrow(w;x,b)$
%given in formula
%\eqref{zdownarrowformula} as $b\downarrow 0$
%is well-defined and it equals
%	\begin{align}
%	Z_\downarrow(w;x,0)&=
%	\frac{w}{\lambda(w+\lambda)}
%	\sum_{k=0}^\infty
%	\left(\frac{\lambda}{w+\lambda}\right)^{k+1}
%	 \sum_{i=0}^{k}
%	 \frac{
%	 1-(1+\tilde{C}(w;0))p^{i-k}
%	 }{ \prod\limits_{j=0,j\neq i}^{k}(1-p^{i-j})}
%	 %
%	%\mathbb{1}_{\{bp^{-k}<x\}}
%\mathrm{e}^{(w+\lambda)p^i x},\label{zdownarrowformula}
%	\end{align}
%for $x>0$ with
%\begin{align}
%\tilde{C}(w;0)&=
%	\frac{
%	\sum\limits_{l=0}^\infty
%	%\mathrm{e}^{-b (w+\lambda) p^{-l}}
%	\frac{\lambda^{l}}{ (w+\lambda)^{l}\prod\limits_{i=1}^{l}(1- p^{-i})}
%	}
%	{\sum\limits_{l=0}^\infty
%	%\mathrm{e}^{-b(w+\lambda)p^{-l} }
%\frac{\lambda^{l}}{p^{l}(w+\lambda)^{l}\prod\limits_{i=1}^{l}(1-p^{-i})}
%	}-1.\label{Eq:ConstantCLSTDown}
%\end{align}
%To show this convergence one can start from the recursion \eqref{recursion}
%with $b=0$, identify its solution and then inverse Laplace transform to get above
%expression as well.
%
\begin{thm}[First passage time for the reflected process at the running infimum]
\label{mainthm7}
The exit time $\widehat\tau_c(x)$ of the process $\widehat{Y}$ that is the reflected version of $X$ reflected at the running supremum satisfies
\begin{itemize}
\item[(i)] $\widehat{\tau}_c(x)=\tau_{\uparrow c}(x)$ when $\underline{X}_0=0$;
\item[(ii)] for $X_0 =x>0$ and $0<X_0 =u\leq x$, instead, the Laplace-Stieltjes transform of $\widehat\tau_c(x)$ equals
	\begin{align}
	\E_x[\mathrm{e}^{-w\widehat{\tau}_c(x)}]&=Z_{\downarrow}(w;x,u)+L_{\uparrow}(w;x,a(w,c,u),u),
	\label{reflinfimum}
	\end{align}
where $a(w,c,u)$ solves \eqref{acb}.
\end{itemize}
\end{thm}

\begin{proof} We will follow the main idea of \cite{Martijnrefl}, although the proof requires substantial changes
compared to the case of L\'evy processes, due to lack of space-homogeneity of our process $X$.
Fix $0<b<x<a$. Recall that $\tau_{a,b}(x)=\min\{\tau_{\uparrow a}(x),\tau_{\downarrow b}(x)\}$ by \eqref{tau-ab-def}. By Theorem \ref{mainthm3},
	\[
	\mathbb{1}_{\{\tau_{\uparrow a}(x)<\tau_{\downarrow b}(x)\}}=L_{\uparrow}(w;X_{\tau_{a,b}(x)}, a,b),
	\]
where we take $L_{\uparrow}(w;x, a,b)=0$ for $x<b$. By the Markov property, we conclude that
	\[
	t\mapsto \E_x[\mathrm{e}^{-w \tau_{a,b}(x)} L_{\uparrow}(w;X_{\tau_{a,b}(x)}, a,b) \mid \mathcal{F}_t]=\mathrm{e}^{-w \tau_{a,b}(x)\wedge t} L_{\uparrow}(w;X_{\tau_{a,b}(x)\wedge t}, a,b)
	\]
is a local martingale.

Similarly, from Theorem \ref{mainthm4}, by putting $Z_\downarrow(w;x,b)=1$ for $x<b$,
we conclude that
	\[
	\mathbb{1}_{\{\tau_{\downarrow b}(x)<\tau_{\uparrow a}(x)\}}=
	Z_\downarrow(w;X_{\tau_{a,b}(x)},b)-L_{\uparrow}(w;X_{\tau_{a,b}(x)},a,b)Z_\downarrow(w;a,b),
	\]
and hence
	\[
	t\mapsto \mathrm{e}^{-w \tau_{a,b}(x)\wedge t}\left(Z_\downarrow(w;X_{\tau_{a,b}(x)\wedge t},b)-L_{\uparrow}(w;X_{\tau_{a,b}(x)\wedge t},a,b)Z_\downarrow(w;a,b)\right)
	\]
is a local martingale as well.

By taking a linear combination, we observe that
	\begin{equation}
	\label{Zmartingale}
	t\mapsto\mathrm{e}^{-w \tau_{a,b}(x)\wedge t}Z_\downarrow(w;X_{\tau_{a,b}(x)\wedge t},b)
	\end{equation}
is also a local martingale. Denoting
\begin{equation}\label{defF}F(w; x, a,b)=Z_\downarrow(w;x,b)+L_{\uparrow}(w;x, a,b),
\end{equation}
this finally means that
	\[
	t\mapsto \mathrm{e}^{-w \tau_{a,b}(x)\wedge t}F(w;X_{\tau_{a,b}(x)\wedge t},b)
	\]
is a local martingale. This is the starting point of our analysis.
\medskip

Since $0<b<a$ are general, we can conclude that
the function $y\rightarrow F(w;y,a,b)$ is in the domain of the extended generator
$\mathcal{A}^\dagger$ of the process $X$ when it is exponentially killed with intensity $w$ (denoted here by $X^\dagger_t$), that is,
the function $F(w;y,a,b)$ is in the set of functions $f$ for which there exists function $\mathcal{A}^\dagger f$ such that the process
$f(X_t^\dagger) -\int_0^t \mathcal{A}^\dagger f(X_s^\dagger)\mathrm{d}s$ is a local martingale.
More precisely, by Dynkin formula
$$\mathrm{e}^{-w \tau_{a,b}(x)\wedge t}F(w;X_{\tau_{a,b}(x)\wedge t},b)
=\mathrm{e}^{-w \tau_{a,b}(x)\wedge t}F(w;X_{\tau_{a,b}(x)\wedge t},b)-\int_0^{\tau_{a,b}(x)\wedge t}
\mathcal{A}^\dagger F(w;X_{s},b) \mathrm{d}s.
$$
That is, for any $0<b<a$,
	\begin{equation}
	\label{martzero}
	\mathcal{A}^\dagger F(w;y,a,b)=
%\mathcal{A}F(w;y,a,b)-wF(w;y,a,b)=
0,\qquad b<y<a.
	\end{equation}
%where $\mathcal{A}f(y)=f^\prime(y)+\lambda (f(yp)-f(y))$ since $X$ is a piecewise-deterministic Markov process; see \cite{Davis} for details.
This completes the first step of our proof.
%Taking $b\downarrow 0$ gives that
%\begin{equation}
%\mathcal{A}Z_\downarrow(w;x,0)-wZ_\downarrow(w;x,0)=0,\qquad x>0
%\end{equation}
%and $Z_\downarrow(w;x,0)=1$ for all $x\leq 0$.

\medskip

In the second step of the proof we use the following version of It\^o's formula (see also Kella and Yor \cite{OfferMarc} adapted to our set-up).  %Let $X^\dagger$ be the process $X$ exponentially killed with intensity $w$.
For a c\`adl\`ag adapted process $t\mapsto V_t$, which is of finite variation, and a function $(y,z)\mapsto f(y, z)$ that is continuous in $y$, that lies in the domain of $\mathcal{A}^\dagger$, and that is continuously differentiable with respect of $z$, we then obtain that
	\eqn{
	\label{loc-mart-dagger}
	t\mapsto f(X_t^\dagger, V_t)-
	\int_0^t
	\mathcal{A}^\dagger f(X_s^\dagger, V_s)\mathrm{d}s
	-
	\int_0^t \frac{\partial}{\partial z} f(X_s^\dagger, z)_{\big|z=V_s} \mathrm{d}V_s^c-
	\sum_{s\leq t}(f(X_s^\dagger, V_s) - f(X_s^\dagger, V_{s-}))
	}
is a local martingale.

Without loss of generality we can assume that $p^k a(w,c,z)\neq z$ for any $k\geq 1$.
Indeed, it is enough to choose appropriate $c$ and then approximate $\widehat{\tau}_c(x)$
for general $c>0$ by the monotonic limit of $\widehat{\tau}_{c_n}(x)$
with respect to $c_n$ satisfying the above condition.
Then we can use \eqref{loc-mart-dagger} with $V_t=\underline{X}_t$ and $f(y,z)=F(w;y,a(w,c,z),z)$
as this function is continuously differentiable with respect of $z$ which follows from the definition of the function $F$ given in
\eqref{defF}, and the formulas \eqref{zdownarrowformula} and \eqref{eq:upward3}-\eqref{eq:upward4}.
Note that
%$X_t-\underline{X}_{t-} = (X_t-\underline{X}_{t})\mathbb{1}_{\{\Delta\underline{X}_{t}=0\}} +
%\Delta\underline{X}_{t}\mathbb{1}_{\{\Delta\underline{X}_{t}<0\}}$ and $X_t-\underline{X}_{t}=0$ if $\Delta\underline{X}_{t}<0$.
%Hence
	\[
	\sum_{s\leq t}\mathrm{e}^{-ws}\left(F(w;X_s, a (w,c,\underline{X}_{s}), \underline{X}_{s})- F(w;X_s, a(w,c,\underline{X}_{s-}), \underline{X}_{s-})\right)
	%\sum_{s\leq t}(Z_\downarrow(w;0,b)- Z_\downarrow(w;\Delta\underline{X}_{s},b))\mathbb{1}_{\{\Delta\underline{X}_{s}<0\}}
	=0.\]
Indeed,
either $X_s>\underline{X}_{s}$ or $X_s=\underline{X}_{s}$ (that is, we crossed the previous infimum at time $s$).
In the first case $\underline{X}_{s}=\underline{X}_{s-}$ and hence
$F(w;X_s, a (w,c,\underline{X}_{s}), \underline{X}_{s})=F(w;X_s, a(w,c,\underline{X}_{s-}), \underline{X}_{s-})$.
Otherwise, $X_s=\underline{X}_{s}\leq \underline{X}_{s-}$ and
$F(w;X_s, a (w,c,\underline{X}_{s}), \underline{X}_{s})=F(w;X_s, a(w,c,\underline{X}_{s-}), \underline{X}_{s-})=1$
because by \eqref{defF} we have $F(w;x, a, z)=1$ for all $x\leq z$.
This follows from the observation that in this case $Z_\downarrow(w;x,z)=1$
by \eqref{Zdownarrowdef1} and
$L_{\uparrow}(w;x, a(w,c,z),z)=0$ by \eqref{eq:upward1}.
Moreover, in our model $\underline{X}_t^c=0$, because of the upward drift and downward jumps of the process $X_t$
we can cross past infimum only by a jump.
By \eqref{martzero} this gives that
	\[
	t\mapsto \mathrm{e}^{-wt} F(w;X_{t%\wedge \tau_{\uparrow a}(x)
	},a(w,c,\underline{X}_{t}),\underline{X}_{t})
	\]
%with $a=a(w,c, \underline{X}_t)$
is a local martingale. Note that up to time $\widehat{\tau}_c(x)$, the processes $X_t$ and $\underline{X}_t$ are bounded by $u+c$.
By the Optional Stopping Theorem, we then obtain that
	\eqn{
	\E_x \left[\mathrm{e}^{-w \widehat{\tau}_c(x)}F(w;X_{\widehat{\tau}_c(x)},a(w,c, \underline{X}_{\widehat{\tau}_c(x)}),\underline{X}_{\widehat{\tau}_c(x)})\right]=F(w;x,a(w,x,u),u).\label{almostfiniteid}
	}
%Since the right-hand side equals $F(w;x,a(w,x,u),u)=Z_{\downarrow}(w;x,u)+L_{\uparrow}(w;x,a(w,c,u),u)$, this is almost \eqref{reflinfimum},
%except for the factor $F(w;X_{\widehat{\tau}_c(x)},a(w,c, \underline{X}_{\widehat{\tau}_c(x)}),\underline{X}_{\widehat{\tau}_c(x)})$ in the %expectation on the left-hand side.

We next argue that $F(w;X_{\widehat{\tau}_c(x)},a(w,c, \underline{X}_{\widehat{\tau}_c(x)}),\underline{X}_{\widehat{\tau}_c(x)})=1$ almost surely.
Observe that $X_{\widehat{\tau}_c(x)}-\underline{X}_{\widehat{\tau}_c(x)}=\widehat{Y}_{\widehat{\tau}_c(x)}=c$,
and hence by the definition of $a(w,x,u)$ in \eqref{acb}, we obtain
	\eqn{
	F(w;X_{\widehat{\tau}_c(x)},a(w, c, \underline{X}_{\widehat{\tau}_c(x)}),\underline{X}_{\widehat{\tau}_c(x)})
	=F(w;\underline{X}_{\widehat{\tau}_c(x)}+c,a(w, c, \underline{X}_{\widehat{\tau}_c(x)}),\underline{X}_{\widehat{\tau}_c(x)})
	=1,
	}
almost surely, as required.
This completes the proof by \eqref{defF} and \eqref{almostfiniteid}.
%The dominated convergence theorem now gives the assertion in Theorem \ref{mainthm7}.
\end{proof}

\section{Discussion and further research}
\label{sec-disc}
In this section, we discuss alternative approaches as well as possible future research.

\paragraph{An alternative approach to the two-sided exit problems in Theorems \ref{mainthm1}--\ref{mainthm4}.}
Two-sided exit problems related to the exit times in \eqref{exittimes_1} and \eqref{exittimes_2}, as studied in Theorems \ref{mainthm1}--\ref{mainthm4}, can also be derived by solving exit problems for an appropriately scaled sequence of queueing models. Consider for example an immigration-and-catastrophe model, in which immigrations occur according to a Poisson process at rate $\beta_m = m\beta$, and catastrophes are governed by the so-called binomial catastrophes mechanism: at the epochs of a catastrophic event that occurs according to an independent Poisson process at rate $\lambda_m=m\lambda$, every member of the population survives with fixed probability %$p$
independently of anything else.
%Namely, the inter-catastrophe intervals
%are assumed to be exponentially distributed with rate $\lambda$ and, given the pre-catastrophe population size $m$; the number of surviving %individuals after the catastrophe follows
%a binomial distribution with parameters $m$ and $p$.
%Thus, the rate of down-jumps from state $m$ to state $m'=0,1,\ldots,m$ is
%	$$
%	\binom{m}{m'}p^{m'}(1-p)^{m-m'}\lambda .
%	$$
%See Figure \ref{figure1} for a visual depiction of the transition rate diagram of the above model and see \cite{Adan-Economou-Kapodistria-2009} and the references therein for further details on these models.
%
%
Denoting
%We consider the sequence of queue length processes
the population size at time $t$ by $Q^{(m)}_t$, one can prove that the exit times of the fluid scaled limit
%	\begin{align*}
$	%\bar{Q}_t=
\lim_{m\to\infty} Q^{(m)}_t/m
$	%\end{align*}
converges weakly to the exit times of the process $X$. We decided to solve our exit problems in a {\em direct} way, thus avoiding additional arguments related to weak convergence.

\paragraph{An alternative approach to our results, with a focus on Theorem \ref{mainthm2}.} We note that a similar approach may be used to obtain the results when, instead of focusing on the first time the process jumps downwards, one focuses on an infinitesimally small time interval right after $0$. We only provide a brief sketch of the derivation and only consider one-sided downward exit time as in Theorem \ref{mainthm2}.

We condition on the number of jumps in $(0,\varepsilon)$. Two jumps will happen with a probability that is $o(\varepsilon)$, one jump with a probability $\lambda \varepsilon + o(\varepsilon)$ and no jumps will happen with probability $1 - \lambda \varepsilon + o(\varepsilon)$.

Denote $\phim = \P(\tau_{\downarrow b}(x) > t)$. Consider first $x \in [b,b/p)$, in which case a jump in $(0,\varepsilon)$ takes the value of the process below $b$ immediately, and therefore
	$$
	\phim = (1-\lambda \varepsilon + o(\varepsilon)) \P(\tau_b^-(x+\varepsilon) > t-\varepsilon) + o(\varepsilon).
	$$
We now divide both side by $\varepsilon$ and let $\varepsilon \nearrow 0$ to obtain
	$$
	\lim_{\varepsilon \to 0}\frac{\phim - \varphi(x+\varepsilon, t-\varepsilon)}{\varepsilon} = - \lambda \phim.
	$$
The numerator on the above LHS may be written as
	$$
	\phim -\varphi(x, t-\varepsilon) + \varphi(x, t-\varepsilon) - \varphi(x+\varepsilon, t-\varepsilon),
	$$
and we therefore obtain
	\begin{equation}
	\label{eq:diff_eq_1}
	\phimx-\phimt = \lambda \phim,
	\end{equation}
which is valid for all $x \in (b,b/p)$. Consider now $x \ge b/p$. For these values,
	$$
	\phim = \lambda \varepsilon \P(\tau_b^-(xp) > t-\varepsilon) +  (1-\lambda \varepsilon + o(\varepsilon)) \P(\tau_b^-(x+\varepsilon > t-\varepsilon) + o(\varepsilon).
	$$
Similar arguments will imply
	\begin{equation}
	\label{eq:diff_eq_2}
	\phimx-\phimt = \lambda \phim - \lambda \varphi(xp,t),
	\end{equation}
which is valid for all $x \ge b/p$. One can then check that the differential equations in \eqref{eq:diff_eq_1} and \eqref{eq:diff_eq_2} are equivalent to the integral equation implied, in a straightforward manner, by \eqref{eq:tau_down}. We refrain from discussing such approaches further.

\paragraph{Applications of our results and future directions. }The exit problems studied in this paper might be used also in applications. An obvious choice is to look at all problems where fluctuation theory has been applied for the L\'evy processes. This is of course a long-term project and we are confident that our results will contribute to its development. Another possible application might lie in the development of {\em asymptotic results}. Indeed, the formulas that we provide for the Laplace transforms of exit times are closely related to the tail behavior of these exit times, through inversion or Tauberian theorems. We refrain from such an analysis, as it requires various different techniques, and thus would make the paper less coherent.

\asymptotics{\section{Asymptotic results}\label{sec-as}

\begin{cor}[A Cram\'er-type estimate]\label{Crameras}
\color{red}
	\begin{equation}\label{Cras}
	\lim_{x\uparrow+\infty}\frac{\log \mathbb{P}_x(\tau_{\downarrow b}(x)\leq t)}
	{\log_{p^{-1}} (x/b) \log \log_{p^{-1}} (x/b)}
	=-1
	\end{equation}
\end{cor}

\todo[inline, color=red]{To be removed unless we prove this using the LST expression}}

\paragraph{Acknowledgments.}
\noindent The work of SK and RvdH is supported by the NWO Gravitation Networks grant 024.002.003. The work of ZP is partially supported by the National Science Centre under the grant  2018/29/B/ST1/00756 (2019-2022).

%\section{Appedix}

%\includepdf[pages={2-},scale=0.95]{derivationofZ.pdf}

%\includepdf[pages={2-},scale=0.95]{derivationof13.pdf}

\end{document}